\documentclass[reqno]{amsart}

\input xy
\xyoption{all}
\usepackage{accents}
\usepackage{epsfig}
\usepackage{color}
\usepackage{amsthm}
\usepackage{amssymb}
\usepackage{amsmath}
\usepackage{amscd}
\usepackage{amsopn}
\usepackage{graphicx}

\usepackage{xspace}

\usepackage{url}
\usepackage{enumitem, hyperref}\hypersetup{colorlinks}

\usepackage{color} 

\definecolor{darkred}{rgb}{1,0,0} 
\definecolor{darkgreen}{rgb}{0,0.8,0}
\definecolor{darkblue}{rgb}{0,0,1}

\hypersetup{colorlinks,
linkcolor=darkblue,
filecolor=darkgreen,
urlcolor=darkred,
citecolor=darkgreen}

\makeatletter
\def\reflb#1#2{\begingroup
    #2%
    \def\@currentlabel{#2}%
    \phantomsection\label{#1}\endgroup
}
\makeatother

\numberwithin{equation}{section}
\newtheorem {Theorem}{Theorem}
\numberwithin{Theorem}{section}

\newtheorem {Lemma}[Theorem]    {Lemma}

\newtheorem {Conjecture}[Theorem]    {Conjecture}
\newtheorem {Proposition}[Theorem]{Proposition}
\newtheorem {Corollary}[Theorem]{Corollary}
\theoremstyle{definition}
\newtheorem{Definition}[Theorem]{Definition}
\theoremstyle{remark}
\newtheorem{Remark}[Theorem]{Remark}
\newtheorem{Example}[Theorem]{Example}

\def    \eps    {\epsilon}

\newcommand{\CA}{{\mathcal A}}
\newcommand{\tCA}{\tilde{\mathcal A}}

\newcommand{\CM}{{\mathcal M}}

\newcommand{\CO}{{\mathcal O}}
\newcommand{\CR}{{\mathcal R}}
\newcommand{\CS}{{\mathcal S}}
\newcommand{\CSI}{{\mathcal S}_{\mathit ind}}

\newcommand{\cCS}{{\check{\mathcal S}}}
\newcommand{\cCSI}{{\check{\mathcal S}}_{\mathit ind}}

\newcommand{\supp}{\operatorname{supp}}

\newcommand{\Ham}{{\mathit{Ham}}}
\newcommand{\tHam}{{\widetilde{\mathit{Ham}}}}

\newcommand{\id}{{\mathit id}}
\newcommand{\pt}{{\mathit pt}}
\newcommand{\const}{{\mathit const}}

\newcommand{\ta}{\tilde{a}}

\newcommand{\hx}{\hat{x}}

\newcommand{\tI}{\tilde{I}}
\newcommand{\tK}{\tilde{K}}
\newcommand{\tB}{\tilde{B}}
\newcommand{\tU}{\tilde{U}}
\newcommand{\tu}{\tilde{u}}
\newcommand{\tv}{\tilde{v}}

\newcommand{\tvarphi}{\tilde{\varphi}}

\newcommand{\tH}{\tilde{H}}

\newcommand{\tA}{\tilde{\mathcal A}}

\newcommand{\PP}{{\mathcal P}}
\newcommand{\bPP}{\bar{\mathcal P}}

\def    \nat    {{\natural}}
\def    \F      {{\mathbb F}}

\def    \C      {{\mathbb C}}
\def    \R      {{\mathbb R}}

\def    \Z      {{\mathbb Z}}
\def    \N      {{\mathbb N}}
\def    \Q      {{\mathbb Q}}

\def    \T      {{\mathbb T}}
\def    \CP     {{\mathbb C}{\mathbb P}}

\def    \12    {{\frac{1}{2}}}

\def    \p      {\partial}
\def    \codim  {\operatorname{codim}}

\def    \Area  {\operatorname{Area}}

\def    \im     {\operatorname{im}}

\def    \SU     {\operatorname{SU}}

\def    \HF     {\operatorname{HF}}

\def    \HQ     {\operatorname{HQ}}

\def    \H     {\operatorname{H}}

\def    \CF     {\operatorname{CF}}

\def    \vol     {\operatorname{vol}}

\def    \bPP     {\bar{\mathcal{P}}}
\def    \bx     {\bar{x}}
\def    \by     {\bar{y}}
\def    \bz     {\bar{z}}

\def    \vp   {\vec{p}}
\def    \vr   {\vec{r}}
\def    \vDelta    {\vec{\Delta}}

\def    \vtheta     {\vec{\theta}}

\def  \hmu      {\operatorname{\hat{\mu}}}

\def    \s  {\operatorname{c}}

\def \inv   {\mathrm{inv}}
\def \qq   {\mathrm{q}}

\def    \chom {\operatorname{c_{\scriptscriptstyle{hom}}}}

\begin{document}


\setlength{\smallskipamount}{6pt}
\setlength{\medskipamount}{10pt}
\setlength{\bigskipamount}{16pt}





\title[Pseudo-Rotations of Projective Spaces]{Hamiltonian Pseudo-Rotations
  of Projective Spaces}

\author[Viktor Ginzburg]{Viktor L. Ginzburg}
\author[Ba\c sak G\"urel]{Ba\c sak Z. G\"urel}

\address{BG: Department of Mathematics, UCF,
  Orlando, FL 32816, USA} \email{basak.gurel@ucf.edu}

\address{VG: Department of Mathematics, UC Santa Cruz, Santa Cruz, CA
  95064, USA} \email{ginzburg@ucsc.edu}



\date{\today} 

\thanks{The work is partially supported by NSF CAREER award
  DMS-1454342 (BG) and NSF grant DMS-1308501 (VG)}


\begin{abstract} The main theme of the paper is the dynamics of
  Hamiltonian diffeomorphisms of $\CP^n$ with the minimal possible
  number of periodic points (equal to $n+1$ by Arnold's conjecture),
  called here Hamiltonian pseudo-rotations. We prove several results
  on the dynamics of pseudo-rotations going beyond periodic orbits,
  using Floer theoretical methods. One of these results is the
  existence of invariant sets in arbitrarily small punctured
  neighborhoods of the fixed points, partially extending a theorem of
  Le Calvez and Yoccoz and of Franks and Misiurewicz to higher
  dimensions. The other is a strong variant of the Lagrangian
  Poincar\'e recurrence conjecture for pseudo-rotations. We also prove
  the $C^0$-rigidity of pseudo-rotations with exponentially Liouville
  mean index vector. This is a higher-dimensional counterpart of a
  theorem of Bramham establishing such rigidity for pseudo-rotations
  of the disk.
\end{abstract}

\maketitle


\tableofcontents


\section{Introduction}
\label{sec:intro+results}

The main theme of the paper is the dynamics of Hamiltonian
diffeomorphisms of $\CP^n$, equipped with the standard symplectic
structure, with exactly $n+1$ periodic points, called here Hamiltonian
pseudo-rotations. By Arnold's conjecture established in this case by
Fortune and Weinstein, \cite{Fo, FW}, and by Floer, \cite{Fl}, this is
the minimal possible number of fixed points and, in particular, of
periodic points.

We prove several results on the dynamics of such maps. One is the
existence of invariant sets in arbitrarily small punctured
neighborhoods of the fixed points, partially extending the results of
Le Calvez and Yoccoz and of Franks and Misiurewicz from \cite{Fr99, FM,
  LCY} to higher dimensions. Another result is a strong variant of the
Lagrangian Poincar\'e recurrence conjecture for pseudo-rotations. We
also prove the $C^0$-rigidity of pseudo-rotations with exponentially
Liouville mean index vector. This is a $2n$-dimensional analog of
Bramham's theorem from \cite{Br}.

Perhaps the most important general point of this paper is that rather
unexpectedly one can obtain a lot information about the dynamics of
pseudo-rotations in dimensions greater than two, going far beyond
periodic orbits, by purely Floer theoretical methods. Before turning
to the precise statements of sample results we need to discuss the
notion of a Hamiltonian pseudo-rotation -- the key player in this work
-- in a greater generality and more detail.

\subsection{Hamiltonian pseudo-rotations}
\label{sec:PRs}
In the framework of two-dimensional dynamics, pseudo-rotations are
area preserving diffeomorphisms of $D^2$ or $S^2$ with exactly one or
two periodic points, respectively, which are then automatically the
fixed points. There are several ways of extending this notion to
higher dimensions in the context of Hamiltonian dynamics and
symplectic topology.

For instance, one can define a Hamiltonian pseudo-rotation of a closed
symplectic manifold $(M^{2n},\omega)$ as a Hamiltonian diffeomorphism
$\varphi$ with finite and minimal possible number of periodic points,
and such that the periodic points are the fixed points. This is the
definition we prefer to use here even though the notion of the minimal
possible number is ambiguous. However, when $M$ is sufficiently nice
and Arnold's conjecture is known to hold for $M$, this can be the sum
of Betti numbers of $M$ when $\varphi$ is non-degenerate, and a
Lusternik--Schnirelmann type lower bound (e.g., the category or the
cup-length plus one) in general. For $\CP^n$ both lower bounds are
$n+1$; hence Definition \ref{def:PR}.

One can also define a pseudo-rotation as a Hamiltonian diffeomorphism
with finitely many periodic points. These are the so-called perfect
Hamiltonian diffeomorphisms studied in the context of the Conley
conjecture; see, e.g., \cite{CGG, GG:survey} and references
therein. As follows from the results in \cite{Fr92, Fr96, FH, LeC}, in
dimension two this definition is equivalent to the one adopted
here. We will further discuss the difference or a lack thereof between
pseudo-rotations and perfect Hamiltonian diffeomorphisms below, but at
the moment we only mention that in all known examples perfect
Hamiltonian diffeomorphisms are non-degenerate pseudo-rotations.

It is believed that rather few manifolds admit
pseudo-rotations. Namely, the Conley conjecture asserts that for a
broad class of closed symplectic manifolds every Hamiltonian
diffeomorphism has infinitely many simple (a.k.a.\ prime, i.e., not
iterated) periodic orbits. At the time of writing, the conjecture has
been established for all symplectic manifolds $(M^{2n},\omega)$ such
that there is no class $A\in\pi_2(M)$ with $\omega(A)>0$ and
$\left<c_1(TM),A\right>>0$; see \cite{GG:Rev}.  In particular, the
conjecture holds for all symplectic CY manifolds, \cite{GG:gaps,
  He:irr}, and all negative monotone symplectic manifolds, \cite{CGG,
  GG:nm}. (See also \cite{FH, Gi:CC, Hi, LeC, SZ} for some milestone
intermediate results and \cite{GG:survey} for a general survey and
further references.) One can also show that $N\leq 2n$ when $M$ admits
a pseudo-rotation, where $N$ is the minimal Chern number of $M$,
although it is not yet known if the Conley conjecture holds whenever
$N>2n$.

On the other hand, all symplectic manifolds $M$ which carry
Hamiltonian circle (or torus) actions with isolated fixed points also
admit pseudo-rotations. Indeed, then a generic element of the circle
or torus induces a pseudo-rotation of $M$. In particular, $\CP^n$, the
complex Grassmannians and flag manifolds, or more generally a majority
of coadjoint orbits of compact Lie groups, and symplectic toric
manifolds admit pseudo-rotations. For $\CP^n$ identified with the
quotient of the unit sphere $S^{2n+1}\subset\C^{n+1}$, such a
pseudo-rotation is a true rotation, by which we mean an element of
$\SU(n)$, and can be generated by a quadratic Hamiltonian
$Q=\sum a_i|z_i|^2$ where $a_i-a_j\not\in\Q$ for $i \neq j$. (The
combinatorics of rotations is not entirely straightforward and we will
examine it more closely in \cite{GG:PR2}; see also Examples
\ref{ex:cpn} and \ref{ex:cpn-2}.)

However, at least for $S^2$, not every pseudo-rotation is a true
rotation.  Indeed, diffeomorphisms arising as generic elements of
$S^1$-actions have simple, essentially trivial dynamics.  This is in
general not the case for pseudo-rotations, and pseudo-rotations occupy
a special place among low-dimensional dynamical systems. For instance,
in \cite{AK} Anosov and Katok constructed area preserving
diffeomorphisms $\varphi$ of $S^2$ with exactly three ergodic
measures, the area form and the two fixed points, by developing what
is now known as the conjugation method; see also \cite{FK} and
references therein. Such a diffeomorphism $\varphi$ is automatically a
pseudo-rotation. Indeed, $\varphi$ is area preserving and hence
Hamiltonian, and $\varphi$ has exactly two periodic orbits, which are
its fixed points.  Furthermore, $\varphi$ is ergodic, necessarily has
dense orbits, and thus is not conjugate to a true rotation.  As a
consequence, the products $(S^2)^n$ also admit pseudo-rotations which
are not conjugate to true rotations.

The situation is more complicated for projective spaces. It is
believed that $\CP^n$ and other toric manifolds admit dynamically
interesting pseudo-rotations, e.g., ergodic or even with finite number
of ergodic measures. However, in the Hamiltonian setting, the
conjugation method encounters a conceptual difficulty along the lines
of the symplectic packing or flexibility questions. This difficulty
disappears for $\CP^2$ due to a theorem of McDuff asserting that any
two symplectic embeddings of any fixed collection of closed balls
into $\CP^2$ are Hamiltonian isotopic; see \cite{McD98,
  McD:ellipsoids}. As a consequence, in this case the problem becomes
more technical than conceptual, although still quite non-trivial due
to the effect of the fixed points. A parallel question in the
symplectomorphism case was studied in \cite{HCP} where a different
variant of the conjugation method introduced in \cite{FaHe} was used
to construct minimal symplectomorphisms of symplectic manifolds
admitting symplectic $S^1$-actions without fixed points.

We conclude this section by elaborating on the conjecture that every
perfect Hamiltonian diffeomorphism, i.e., a Hamiltonian diffeomorphism
with finitely many periodic points, is automatically a
pseudo-rotation, i.e., it has the minimal possible number of periodic
points, and all such points are fixed points. (Moreover,
hypothetically every perfect Hamiltonian diffeomorphism is strongly
non-degenerate and all its periodic orbits, which are then fixed
points, are elliptic.) For instance, for $\CP^n$, according to this
conjecture every Hamiltonian diffeomorphism with more than $n+1$ fixed
points must have infinitely many periodic orbits. As has been pointed
out above, in dimension two this fact is essentially the combination
of the celebrated theorem of Franks (see \cite{Fr92, Fr96, LeC})
asserting that every area preserving diffeomorphism of $S^2$ with more
than two fixed points has infinitely many periodic orbits and the
Conley conjecture for surfaces proved in \cite{FH}. (See also
\cite{CKRTZ} for a Floer theoretical proof of Franks' theorem and
\cite{BH} for an approach utilizing finite energy foliations.) The
earliest mentioning of this hypothetical generalization of Franks'
theorem known to us is on p.\ 263 in~\cite{HZ}.

One can also make similar conjectures for symplectomorphisms and other
types of ``Hamiltonian dynamical systems''. Furthermore, looking at
the question from a broader perspective one may expect that the
presence of a periodic orbit which is homologically or geometrically
unnecessary (e.g., degenerate, non-contractible, hyperbolic) forces
the system to have infinitely many periodic orbits. We refer the
reader to \cite{Ba1, Ba2, GG:hyperbolic, GG:nc, Gu:nc, Gu:hq, Or1,
  Or2} (and also to Theorem \ref{thm:isolated}) for some recent
results in this direction in dimensions greater than~two.

\subsection{Sample results}
\label{sec:results} In this section, we state without detailed
discussion and in some instances in a simplified form the key results
of the paper. This is a ``sampler plate'' and a much more thorough
treatment is given in Sections \ref{sec:sets} and
\ref{sec:PR+PO}. Here and throughout the paper, $\CP^n$ is equipped
with the standard (sometimes up to a factor) Fubini--Study symplectic
structure.

\begin{Definition}
\label{def:PR}
A pseudo-rotation of $\CP^n$ is a Hamiltonian diffeomorphism of $\CP^n$
with exactly $n+1$ periodic points.
\end{Definition}

Our first theorem concerns invariant sets of pseudo-rotations. As has
been mentioned above, the pseudo-rotations $\varphi$ of $S^2$
constructed by the conjugation method in \cite{AK} have exactly three
ergodic measures: the two fixed points and the area form. As a
consequence, such pseudo-rotations are very close to being uniquely
ergodic: $\varphi$ is uniquely ergodic on the complement to the two
fixed points. Recall also that volume preserving uniquely ergodic maps
of compact manifolds are necessarily minimal, i.e., all orbits are
dense; see, e.g., \cite{Wa}. These facts suggest that every orbit of
$\varphi$ other than a fixed point should be dense. (Furthermore,
minimal symplectomorphisms exist in abundance and can also be
constructed by the conjugation method; \cite{HCP}.) However, this
turns out to be false and $\varphi$ must have many proper, closed
invariant subsets as proved in \cite[Prop.\ 5.5]{FM}, which is in turn
inspired by \cite{LCY} and based on the approach developed in
\cite{Fr99}. Our first theorem is a partial generalization of these
results to higher dimensions:

\begin{Theorem}
\label{thm:inv_sets0}
No fixed point of a pseudo-rotation of $\CP^n$ is isolated as an
invariant set.
\end{Theorem}

This theorem is proved in Section \ref{sec:sets}. The proof hinges on
a result of independent interest asserting that a Hamiltonian
diffeomorphism of $\CP^n$ with a fixed point which is isolated as an
invariant set and has non-vanishing local Floer homology must have
infinitely many periodic orbits. This is Theorem \ref{thm:isolated}
extending \cite[Thm.\ 1.1]{GG:hyperbolic} to degenerate periodic
orbits and proved in Section \ref{sec:energy-isolated}.
 
Our next result concerns the Lagrangian Poincar\'e recurrence
conjecture put forth by the first author and independently by Viterbo
around 2010. According to this conjecture the images of a Lagrangian
submanifold $L$ under the iterations of a Hamiltonian diffeomorphism
with compact support cannot be all disjoint. In dimension two, this is
an easy consequence of the standard Poincar\'e recurrence combined
with the observation that when $L$ does not bound, it cannot be
displaced by a Hamiltonian diffeomorphism. However, in higher
dimensions the assertion is not obvious even for a particular
Hamiltonian diffeomorphism unless, of course, it is periodic. We
discuss Lagrangian Poincar\'e recurrence in more detail in Section
\ref{sec:PR}, where we establish the conjecture for pseudo-rotations
of $\CP^n$ and a sufficiently broad class of Lagrangians $L$. In
particular, we prove the following.

\begin{Theorem}
\label{thm:LPR0}
Let $\varphi$ be a pseudo-rotation of $\CP^n$ and let $L\subset \CP^n$
be a closed Lagrangian submanifold admitting a metric without
contractible closed geodesics. Then $\varphi^k(L)\cap L\neq \emptyset$
for infinitely many $k\in \N$.
\end{Theorem}

This theorem is a consequence of Corollary \ref{sec:gamma}
($\gamma$-norm convergence) asserting that
$\gamma(\varphi^{k_i})\to 0$ for some sequence $k_i\to\infty$, where
$\gamma$ is the $\gamma$-norm. In fact, in Theorem \ref{thm:LPR} we
establish a stronger version of the Lagrangian Poincar\'e recurrence
conjecture relating the return rate, i.e., the frequency of
intersections $\varphi^k(L)\cap L$, and the homological capacity
$\chom(L)$ of $L$. Moreover, the result holds for all compact subsets
$L$ with $\chom(L)>0$. The proof is based on a quantitative version of
the $\gamma$-norm convergence; see Section \ref{sec:gamma} and Theorem
\ref{thm:gamma}.

Finally, our third result is $C^0$-rigidity of pseudo-rotations of
$\CP^n$. This is a higher-dimensional counterpart of the main theorem
of \cite{Br}.  Let $\varphi$ be a Hamiltonian pseudo-rotation of
$\CP^n$ with fixed points $x_0,\ldots,x_n$. These are also its
periodic points. The mean index ``vector'' of $\varphi$ is by
definition
$$
\vDelta=\big(\hmu(x_0),\ldots,\hmu(x_n)\big)\in
\T^{n+1}=\R^{n+1}/2(n+1)\Z^{n+1},
$$
where we treat the mean indices $\hmu(x_i)$ of uncapped one-periodic
orbits as elements of the circle $\R/ 2(n+1)\Z$; see Section
\ref{sec:background} for more details. We say that $\vDelta$ is
exponentially Liouville if the iterates $k\vDelta$ approximate
$0\in \T^{n+1}$ exponentially accurately: for any $c>0$ there exists
$k\in \N$ such that $\|k\vDelta\|<e^{-ck}$, where $\|\cdot\|$ is the
distance to $0$; see Definition \ref{def:exp-L}. In $\T^{n+1}$ or any
closed subgroup of $\T^{n+1}$, exponentially Liouville elements form a
zero measure, residual set (i.e., a countable intersection of open and
dense sets).

\begin{Theorem}
\label{thm:exp-L0}
For every pseudo-rotation $\varphi$ of $\CP^n$ with exponentially
Liouville mean index vector $\vDelta$, there exists a sequence
$k_i\to\infty$ such that
$\varphi^{k_i}\stackrel {C^0}{\longrightarrow}\id$.
\end{Theorem}

This theorem is proved in Section \ref{sec:PO}. Note that here we do
not impose any non-degeneracy requirements on $\varphi$. For $\CP^1$,
Theorem \ref{thm:exp-L0} is an easy consequence of \cite[Thm.\ 1]{Br}
where a similar result is established for pseudo-rotations of $D^2$;
see the end of Section \ref{sec:PO} for further discussion.

Although the proofs of these theorems are formally independent of each
other, the arguments share some common components. One of these,
directly entering the proofs of Theorems \ref{thm:LPR0} and
\ref{thm:exp-L0}, is a variant of Lusternik--Schnirelmann theory for
pseudo-rotations going back to \cite{GG:gaps}.  To state the key
result, note that every point in the action spectrum $\CS(H)$ of a
pseudo-rotation $\varphi_H$ of $\CP^n$ is the value of an action
selector (a spectral invariant) for some homology class in the quantum
homology $\HQ_*(\CP^n)$. Then, by Theorem \ref{thm:bijection}, the
ordering of $\CS(H)$ by the degree or, equivalently, by the
Conley--Zehnder index when $\varphi_H$ is non-degenerate agrees with
its natural ordering as a subset of $\R$. In particular, all action
values are distinct. Furthermore, $\CS(H)$ coincides with the mean
index spectrum up to scaling and a shift; see Theorem
\ref{thm:act-index} and identities \eqref{eq:act-index2} and
\eqref{eq:equal-spectra}.  We review these and other previously known
results on symplectic topology of pseudo-rotations in Section
\ref{sec:background}. The notation, conventions and general
preliminary facts are discussed in Section~\ref{sec:prelim}.

\medskip

\noindent{\bf Acknowledgements.} The authors are grateful to Barney
Bramham, Weinmin Gong, Felix Schlenk, Sobhan Seyfaddini, Egor
Shelukhin, Michael Usher, Joa Weber and the referees for useful
remarks and discussions. Parts of this work were carried out while
both of the authors were visiting the Lorentz Center (Leiden,
Netherlands) for the \emph{Hamiltonian and Reeb Dynamics: New Methods
  and Applications} workshop, during the first author's visit to NCTS
(Taipei, Taiwan) and the second author's visit to UCSC (Santa Cruz,
California).  The authors would like to thank these institutes for
their warm hospitality and support.

\section{Preliminaries}
\label{sec:prelim}
The goal of this section is to set notation and conventions, give a
brief review of Hamiltonian Floer homology and several other notions
from symplectic geometry used in the paper. We also state and prove
several facts (e.g., Lemmas \ref{lemma:spec} and \ref{lemma:trivial})
which, although quite standard, seem to be unavailable elsewhere.  The
reader may consider consulting this section only as necessary.

\subsection{Conventions and notation}
\label{sec:conventions}
Let $(M^{2n},\omega)$ be a closed symplectic manifold. Throughout most
of the paper, except parts of Sections \ref{sec:sets} and
\ref{sec:energy-isolated}, this manifold will be $\CP^n$ equipped with
the standard symplectic form $\omega$. It will be convenient however
to vary the normalization of $\omega$ and we set
$\lambda=\left<[\omega],\CP^1\right>$. In other words, $\lambda$ is
the positive generator of $\left< [\omega],\pi_2(M)\right>\subset \R$,
the rationality constant. For the standard Fubini--Study normalization
$\lambda=\pi$. Recall also that the minimal Chern number $N$, i.e.,
the positive generator of $\left< c_1(TM),\pi_2(M)\right>$, is $n+1$
for $\CP^n$. On several occasions we will work with more general
symplectic manifolds. Then $M$ is always assumed to be \emph{strictly
  monotone}, i.e.,
$[\omega]\!\mid_{\pi_2(M)}=\tau c_1(TM)\!\mid_{\pi_2(M)}\neq 0$ for
some $\tau\in\R$, and hence $M$ is rational and $N<\infty$. In
particular, $\tau=\lambda/N$.

All Hamiltonians $H$ considered in this paper are assumed to be
$k$-periodic in time, i.e., $H\colon S^1_k\times M\to\R$, where
$S^1_k=\R/k\Z$ and $k\in\N$.  When the period $k$ is not specified, it
is equal to one and $S^1=S^1_1=\R/\Z$. We set $H_t = H(t,\cdot)$ for
$t\in S^1_k$. The Hamiltonian vector field $X_H$ of $H$ is defined by
$i_{X_H}\omega=-dH$. The (time-dependent) flow of $X_H$ is denoted by
$\varphi_H^t$ and its time-one map by $\varphi_H$. Such time-one maps
are referred to as \emph{Hamiltonian diffeomorphisms}.  A one-periodic
Hamiltonian $H$ can always be treated as $k$-periodic, which we will
then denote by $H^{\nat k}$ and, abusing terminology, call
$H^{\nat k}$ the $k$th iterate of $H$.

Let $H$ and $K$ be one-periodic Hamiltonians such that $H_1=K_0$
together with $t$-derivatives of all orders. We denote by $H \nat K$
the two-periodic Hamiltonian equal to $H_t$ for $t\in [0,\,1]$ and
$K_{t-1}$ for $t\in [1,\,2]$. Thus $H^{\nat k}=H \nat \ldots \nat H$
($k$ times). More generally, when $H$ is $l$-periodic and $K$ is
$k$-periodic, $H\nat K$ is $(l+k)$-periodic. (Strictly speaking, here
we need to assume that $H_l=K_0$ again together with all
$t$-derivatives.)

Let $x\colon S^1_k\to M$ be a contractible loop. A \emph{capping} of
$x$ is an equivalence class of maps $A\colon D^2\to M$ such that
$A\mid_{S^1_k}=x$. Two cappings $A$ and $A'$ of $x$ are equivalent if
the integrals of $\omega$ and $c_1(TM)$ over the sphere obtained by
attaching $A$ to $A'$ are equal to zero. A capped closed curve
$\bar{x}$ is, by definition, a closed curve $x$ equipped with an
equivalence class of cappings. In what follows, the presence of
capping is always indicated by a bar.

The action of a Hamiltonian $H$ on a capped closed curve
$\bar{x}=(x,A)$ is
$$
\CA_H(\bar{x})=-\int_A\omega+\int_{S^1} H_t(x(t))\,dt.
$$
The space of capped closed curves is a covering space of the space of
contractible loops, and the critical points of $\CA_H$ on this space
are exactly the capped one-periodic orbits of $X_H$. The \emph{action
  spectrum} $\CS(H)$ of $H$ is the set of critical values of
$\CA_H$. This is a zero measure set; see, e.g., \cite{HZ}.

The $k$-periodic \emph{points} of $\varphi_H$ are in one-to-one
correspondence with the $k$-periodic \emph{orbits} of $H$, i.e., of
the time-dependent flow $\varphi_H^t$, which we denote by
$\PP_k(H)$. (Thus, for instance, $\PP_1(H)$ can be identified with the
fixed point set of $\varphi_H$.)  The collections of all periodic
orbits and all capped periodic orbits of $H$ will be denoted by
$\PP(H)$ and, respectively, $\bPP(H)$. Recall also that a $k$-periodic
orbit of $H$ is called \emph{simple} or \emph{prime} if it is not
iterated. The definition of the action of $H$ extends to $k$-periodic
orbits and Hamiltonians in an obvious way. Clearly, the action
functional is homogeneous with respect to iteration:
$$
\CA_{H^{\nat k}}\big(\bx^k\big)=k\CA_H(\bx),
$$
where $\bx^k$ is the $k$th iteration of the capped orbit $\bx$. (The
capping of $\bx^k$ is obtained from the capping of $\bx$ by taking its
$k$-fold cover branched at the origin.)

A $k$-periodic orbit $x$ of $H$ is said to be \emph{non-degenerate} if
the linearized return map $d\varphi_H^k \colon T_{x(0)}M\to T_{x(0)}M$
has no eigenvalues equal to one. Following \cite{SZ}, we call $x$
\emph{weakly non-degenerate} if at least one of the eigenvalues is
different from one and \emph{totally degenerate} if all eigenvalues
are equal to one. A Hamiltonian $H$ is (weakly) non-degenerate if all
its one-periodic orbits are (weakly) non-degenerate and $H$ is
\emph{strongly non-degenerate} if all periodic orbits of $H$ (of all
periods) are non-degenerate.

Let $\bar{x}=(x,A)$ be a non-degenerate capped periodic orbit.  The
\emph{Conley--Zehnder index} $\mu(\bar{x})\in\Z$ is defined, up to a
sign, as in \cite{Sa,SZ}. In this paper, we normalize $\mu$ so that
$\mu(\bar{x})=n$ when $x$ is a non-degenerate maximum (with trivial
capping) of an autonomous Hamiltonian with small Hessian. The
\emph{mean index} $\hmu(\bx)\in\R$ measures, roughly speaking, the
total angle swept by certain unit eigenvalues of the linearized flow
$d\varphi^t_H|_x$ with respect to the trivialization associated with
the capping; see \cite{Lo,SZ}. The mean index is defined even when $x$
is degenerate and depends continuously on $H$ and $\bx$ in the obvious
sense.  Furthermore,
\begin{equation}
\label{eq:mean-cz}
\big|\hmu(\bx)-\mu(\bx)\big|\leq n
\end{equation} 
and the inequality is strict when $x$ is weakly non-degenerate.  The
mean index is homogeneous with respect to iteration:
$\hmu\big(\bx^k\big)=k\hmu(\bx)$. For an uncapped orbit $x$, the mean
index $\hmu(x)$ and the action $\CA_H(x)$ are well defined as elements
of $S^1_{2N}=\R/2N\Z$ and, respectively, $S^1_\lambda=\R/\lambda\Z$;
see \eqref{eq:recap} and \eqref{eq:delta}. Likewise, when $x$ is
non-degenerate, the Conley--Zehnder index $\mu(x)$ is well defined as
an element of $\Z/2N\Z$.

\subsection{Global and local Floer homology}
\label{sec:FH}
In this subsection, we very briefly discuss, mainly to set notation,
the constructions of filtered and local Floer homology. We refer the
reader to, e.g., \cite{GG:gaps, MS, Sa, SZ} for detailed accounts and
additional references.

\subsubsection{Filtered Floer homology}
\label{sec:FFH}
Fix a ground field $\F$. Let $H$ be a non-degenerate Hamiltonian on
$M$. Denote by $\CF^{(-\infty,\, b)}_m(H)$, with
$b\in (-\infty,\,\infty]\setminus\CS(H)$, the vector space of formal
finite linear combinations
$$ 
\sigma=\sum_{\bar{x}\in \bPP(H)} \sigma_{\bar{x}}\bar{x}, 
$$
where $\sigma_{\bar{x}}\in\F$ and $\mu(\bar{x})=m$ and
$\CA(\bx)<b$. (Since $M$ is strictly monotone, and thus $N<\infty$, we
do not need to take a completion here.) The graded $\F$-vector space
$\CF^{(-\infty,\, b)}_*(H)$ is endowed with the Floer differential
counting solutions $u\colon \R\times S^1\to M$ of the Floer equation,
$$
\p_s u+J\p_t u=-\nabla H_t(u),
$$
where $(s,t)$ are the coordinates on $\R\times S^1$, i.e., the
$L^2$-anti-gradient trajectories of the action functional. Thus we
obtain a filtration of the total Floer complex
$\CF_*(H):=\CF^{(-\infty,\, \infty)}_*(H)$. Furthermore, set
$I=(a,\,b)$ and
$$
\CF^{I}_*(H):=\CF^{(-\infty,\,
  b)}_*(H)/\CF^{(-\infty,\,a)}_*(H),
$$
where $-\infty\leq a<b\leq\infty$ are not in $\CS(H)$. This is simply
the complex generated by the capped periodic orbits of $H$ with action
in $I$ and equipped with the Floer differential.  The resulting
homology, the \emph{filtered Floer homology} of $H$, is denoted by
$\HF^{I}_*(H)$ and by $\HF_*(H)$ when $I=(-\infty,\,\infty)$.  Working
with filtered Floer homology, \emph{we will always assume that the end
  points of the action interval are not in the action spectrum.} The
degree of a class $\alpha\in \HF^{(a,\, b)}_*(H)$ is denoted
by~$|\alpha|$.

Also recall that the energy $E(u)$ of a solution $u$ of the Floer
equation is defined by
$$
E(u)=\int_{\R\times S^1}\|\p_s u\|^2\,ds\,dt=\int_{-\infty}^\infty
\|\p_s u(s)\|^2_{L^2}\,ds
$$
and that 
$$
E(u)=\CA_H(\bx)-\CA_H(\by)
$$
when $u$ is asymptotic to $\bx$ at $-\infty$ and to $\by$ at $\infty$.

The total Floer complex and homology are modules over the
\emph{Novikov ring} $\Lambda$. In this paper, the latter is defined as
follows. Set
$$
I_\omega(A)=-\omega(A)\text{ and } I_{c_1}(A)=-2\left<c_1(TM),
  A\right>,
$$
where $A\in\pi_2(M)$.  Thus $I_\omega=\tau I_{c_1}/2$ since $M$ is
strictly monotone.

Let $\Gamma$ be the quotient of $\pi_2(M)$ by the equivalence relation
where the two spheres $A$ and $A'$ are considered to be equivalent if
$I_\omega(A)=I_\omega(A')$ or, equivalently, since $M$ is strictly
monotone, $I_{c_1}(A)=I_{c_1}(A')$. Clearly, $\Gamma\simeq \Z$. The
homomorphisms $I_\omega$ and $I_{c_1}$ descend to $\Gamma$ and become
isomorphisms onto the image.

The group $\Gamma$ acts on $\CF_*(H)$ and $\HF_*(H)$ by recapping: an
element $A\in \Gamma$ acts on a capped one-periodic orbit $\bar{x}$ of
$H$ by attaching the sphere $A$ to the original capping. Denoting the
resulting capped orbit by $\bx\# A$, we have
\begin{equation}
\label{eq:recap}
\mu(\bx\# A)=\mu(\bx)+ I_{c_1}(A)
\end{equation}
when $x$ is non-degenerate. In a similar vein,
\begin{equation}
\label{eq:delta}
\CA_H(\bx\# A)=\CA_H(\bx)+I_\omega(A)
\text{ and } \hmu(\bx\# A)=\hmu(\bx)+ I_{c_1}(A)
\end{equation}
regardless of whether $x$ is non-degenerate or not.

In general, the Novikov ring $\Lambda$ is a certain completion of the
group ring of $\Gamma$ over $\F$ but for our purposes, since $M$ is
monotone, we can take as $\Lambda$ the group ring itself, i.e., the
collection of finite linear combinations $\sum \alpha_A e^A$, where
$\alpha_A\in\F$ and $A\in \Gamma$. The Novikov ring $\Lambda$ is
graded by setting $|e^A|=I_{c_1}(A)$.  The action of $\Gamma$ turns
$\CF_*(H)$ and $\HF_*(H)$ into $\Lambda$-modules. The maps $I_{c_1}$
and $I_\omega$ naturally extend to $\Lambda$. We set $\qq=e^{A_0}$,
where $A_0$ is the generator of $\Gamma$ with $I_{c_1}(A_0)=-2N$. Thus
$|\qq|=-2N$ and $I_\omega(\qq)=-\lambda$.

The construction of the filtered Floer homology extends to all, not
necessarily non-degenerate, Hamiltonians by continuity.  Namely, let
$H$ be an arbitrary (one-periodic in time) Hamiltonian on $M$ and let
the end-points $a$ and $b$ of the action interval $I$ be outside
$\CS(H)$. By definition, we set
$$
\HF^{I}_*(H)=\HF^{I}_*(\tH),
$$
where $\tH$ is a non-degenerate, small perturbation of $H$. The
right-hand side is independent of $\tH$ once $\tH$ is sufficiently
close to $H$.  The total Floer homology is independent of the
Hamiltonian and, up to a shift of the grading and the effect of
recapping, is isomorphic to the homology of $M$. More precisely, we
have
\begin{equation}
  \label{eq:H-M}
\HF_*(H)\cong \H_ {*+n}(M;\F)\otimes \Lambda
\end{equation}
as graded $\Lambda$-modules.

\begin{Example}[Projective Spaces]
  Let $H$ be a Hamiltonian on $\CP^n$. Then, by \eqref{eq:H-M},
  $\HF_m(H)=\F$ when $m$ has the same parity as $n$ and $\HF_m(H)=0$
  otherwise. To see this directly, one can, for instance, take a
  non-degenerate quadratic Hamiltonian as $H$ and show by a
  calculation that all fixed points of $\varphi_H$ are elliptic and
  hence their indices have the same parity as $n$. (In particular, the
  Floer differential vanishes.)  Recapping by the generator $A_0$ of
  $\Gamma\cong\Z$ decreases the index by $2N=2(n+1)$.
\end{Example}

\subsubsection{Local Floer homology}
\label{sec:LFH}
The notion of \emph{local Floer homology} goes back to the original
work of Floer and it has been revisited a number of times since
then. Here we only briefly recall the definition following mainly
\cite{Gi:CC,GG:gaps,GG:gap} where the reader can find a much more
thorough discussion and further references.

Let $x$ be an isolated one-periodic orbit of a Hamiltonian
$H\colon S^1\times M\to \R$. The local Floer homology $\HF(x)$ is the
homology of the Floer complex $\CF_*(\tH,x)$ generated by the orbits
$x_i$ which $x$ splits into under a $C^2$-small non-degenerate
perturbation $\tH$ of $H$. This homology is well defined, i.e.,
independent of the perturbation. The homology $\HF(x)$ is only
relatively graded and to fix an absolute grading one needs to pick a
trivialization of $TM$ along $x$. This can be done by using, for
instance, a capping of $x$ and in this case we write
$\HF_*(\bx)$. Note that then the orbits $x_i$ inherit a capping from
$\bx$.

For example, if $x$ is non-degenerate and $\mu(\bx)=m$, we have
$\HF_l(\bx)=\F$ when $l=m$ and $\HF_l(\bx)=0$ otherwise. This
construction is local: it requires $H$ to be defined only on a
neighborhood of $x$.

By definition, the \emph{support} of $\HF_*(\bx)$, denoted by
$\supp\HF_*(\bx)$, is the collection of integers $m$ such that
$\HF_m(\bx)\neq 0$. By \eqref{eq:mean-cz} and continuity of the mean
index,
\begin{equation}
\label{eq:supp}
\supp \HF_*(\bx)\subset [\hmu(\bx)-n,\, \hmu(\bx)+n].
\end{equation}
Moreover, when $x$ is weakly non-degenerate, the closed interval can
be replaced by the open interval.

The local Floer homology groups are building blocks for the filtered
Floer homology. For instance, assume that $c$ is the only point of
$\CS(H)$ in an interval $I$ and that all one-periodic orbits of $H$
with action $c$ are isolated. Then, as is easy to see,
\begin{equation}
\label{eq:split}
\HF_*^I(H)=\bigoplus_{\CA_H(\bx)=c}\HF_*(\bx).
\end{equation}
To prove this and also with Lemma \ref{lemma:spec} in mind, first
observe that, when the orbits are isolated, a finite energy solution
of the Floer equation $u$ is automatically asymptotic to unique
one-periodic orbits as $s\to\pm\infty$ and there is an \emph{a priori}
lower bound $\eps$ on the energy of $u$; see, e.g., \cite[Sect.\
1.5]{Sa}. (Hence, we have, in this case, the notion of a solution
connecting two orbits.)  Then one can shrink $I$ to an interval with
length below $\eps$ and notice that for a sufficiently small
non-degenerate perturbation $\tH$ the Floer complex splits; see
\cite[Sect.\ 2.5]{GG:gaps}.

We have the following simple result sharpening \eqref{eq:split}, which
we state and prove in a slightly more general form than needed
here. Denote by $d(\cdot,\,\cdot)$ the distance between two subsets of
$\R$, i.e.,
$$ 
d(A,\,B)=\inf\{|a-b|\mid a\in A,\, b\in B\}.
$$

\begin{Lemma}
\label{lemma:spec}
Assume that all capped one-periodic orbits $\bx$ of $H$ with action in
$I$ are isolated. Then there exists a spectral sequence with
\begin{equation}
\label{eq:E1}
E^1=\bigoplus_{\bx}\HF_*(\bx)
\end{equation}
converging to $\HF_*^I(H)$. Furthermore, $\HF_*(\bx)$ is a direct
summand in $\HF_*^I(H)$ when one of the following two conditions is
met:
\begin{itemize}
\item[(i)] $\bx$ is not connected to any other capped periodic orbit
  of $H$ with action in $I$ by a solution of the Floer equation;
\item[(ii)] for any $\by$ with $\CA_H(\by)\in I$, we have
$$
d\big(\supp\HF_*(\bx), \supp\HF_*(\by)\big)>1.
$$
\end{itemize}
Moreover, assume that (i) holds or the following condition is
satisfied:
\begin{itemize}
\item[(ii')] for any $\by$ with $\CA_H(\by)\in I$, we have
$$
\big|\hmu(\bx)-\hmu(\by)\big|> 2n+1.
$$
\end{itemize}
Then, when $\tH$ is sufficiently $C^2$-close to $H$, the complex
$\CF_*(\tH,\bx)$ is a direct summand in $\CF_*^I(\tH)$.
\end{Lemma}

Note that while requirement (ii') is more restrictive than (ii), the
third assertion of the lemma, that $\CF_*(\tH,\bx)$ enters
$\CF_*^I(\tH)$ as a direct summand, is stronger than the second
concerning a similar fact on the level of homology.  The proof of the
lemma is routine and below we merely comment on the argument.

\begin{proof}[Outline of the proof]
  The required spectral sequence is associated with the action
  filtration on the Floer complex of a small non-degenerate
  perturbation of $H$ and \eqref{eq:E1} is essentially a rephrasing of
  \eqref{eq:split}. (It is a variant of the Morse--Bott spectral
  sequence.) To prove the second assertion, it is enough to show that
  the parts of the differentials $d_j$, $j\geq 2$, connecting
  $\HF_*(\bx)$ with the rest of the spectral sequence vanish when (i)
  or (ii) holds. When (i) is satisfied, this is a consequence of the
  third assertion proved below. When (ii) holds, the vanishing of the
  differentials follows from the fact that $|\mu(\bx_i)-\mu(\by_j)|>1$
  for all capped orbits $\bx_i$ and $\by_j$ which $\bx$ and,
  respectively, $\by$ split into. This argument also shows that
  $\CF_*(\tH,\bx)$ is a direct summand when (ii') is satisfied.

  Finally, let us prove that $\CF_*(\tH,\bx)$ is a direct summand in
  $\CF_*^I(\tH)$ when (i) holds. Arguing by contradiction assume that
  there is a sequence of Floer trajectories $u$ connecting $\bx_i$ and
  $\by_j$ for some sequence of perturbations $\tH\to H$. (Say, $u$ is
  asymptotic to $\bx_i$ at $-\infty$ and $\by_j$ at $\infty$.) Fix a
  closed neighborhood $U$ of $x$ which contains no other one-periodic
  orbits of $H$. Without loss of generality we may assume that
  $u(0,0)\in\p U$ and that the half-cylinder
  $(-\infty,\, 0]\times S^1$ is mapped into $U$. Applying the
  target-local compactness theorem from \cite{Fi} to the restrictions
  of $u$ to a finite cylinder $[-L,\,L]\times S^1$ and then the
  diagonal process as $L\to\infty$, we obtain a solution $v$ of the
  Floer equation for $H$ mapping the half-cylinder
  $(-\infty,\, 0]\times S^1$ into $U$. Hence, $v$ is asymptotic to
  $\bx$ at $-\infty$. It is easy to see that the capped orbit which
  $v$ is asymptotic to at $\infty$ has action in $I$. This contradicts
  (i).
\end{proof}

\subsection{Cap product}
\label{sec:cap-product}
The algebraic structure on the Floer homology, which is crucial for
what follows, is that of a module over the (small) \emph{quantum
  homology} of $M$. The quantum homology $\HQ_*(M)$ of $M$ is an
algebra over the Novikov ring $\Lambda$ defined above and
$\HQ_*(M)=\H_*(M)[-n]\otimes \Lambda$ as $\Lambda$-modules. We refer
the reader to, e.g., \cite[Chap.\ 11]{MS} for the definition of the
product $*$ in $\HQ_*(M)$.  The fundamental class $[M]$ is the unit
with respect to this product. The maps $I_{c_1}$ and $I_\omega$
naturally extend to $\HQ_*(M)$. For instance,
$$
I_\omega(\alpha) =\max\,\{I_\omega(A)\mid \alpha_A\neq 0\}
=\max\,\{-\lambda_0 k \mid \alpha_k\neq 0\},
$$
where $\alpha=\sum \alpha_A e^A=\sum \alpha_k \qq^k\in \HQ_*(M)$.

\begin{Example} 
\label{ex:cpn}
We will make an extensive use of the product structure
on $\HQ_*(\CP^n)$. In this case $N=n+1$ and $\HQ_*(\CP^n)$ is generated by
the hyperplane class $[\CP^{n-1}]$. To be more precise, we have
$[\CP^{n-1}]^{\ell}=[\CP^{n-\ell}]$ when $\ell\leq n$ and
\begin{equation}
\label{eq:qp-cpn}
[\CP^{n-1}]^{n+1}=\qq [\CP^n],
\end{equation}
where $|\qq|=-2(n+1)$; see, e.g., \cite[Sect.\ 11.3]{MS} and
references therein. In other words,
$[\CP^{n-1}] * \alpha_\ell=\alpha_{\ell-1}$, where $\alpha_\ell$ is a generator
of $\HQ_{2\ell-n}(\CP^n)$. For instance, $[\pt]*[\CP^{n-1}]=\qq
[\CP^n]$, which reflects the fact that there is exactly one
line passing through any two distinct points of $\CP^n$ and when 
a point $\pt$ is fixed every line through $\pt$ intersects
$\CP^{n-1}\not\ni\pt$ exactly once. 
\end{Example}

Identifying the global Floer homology with $\HQ_*(M)$, we can view
$\HF_*(H)$ as an $\HQ_*(M)$-module. It is important for our purposes
that this module structure extends in a certain way to the filtered
Floer homology, and we refer to the resulting ``$\HQ_*(M)$-action'',
\eqref{eq:cap-action}, as the \emph{cap product}. (Strictly speaking,
the cap product is not a true algebra action due to a filtration
shift; rather it should be thought of as an ``action'' of $\HQ_*(M)$
on the entire collection of the filtered Floer homology groups.)  The
definition of the cap product, recalled below, goes back to
\cite{Vi:product} and \cite{LO}. Here we closely follow \cite[Sect.\
2.3]{GG:hyperbolic}; see also \cite[Rmk.\ 12.3.3]{MS}.

On the level of cycles, the action of a pseudo-cycle $\zeta$ with
$[\zeta]\in \H_\ell(M)$ on $\bx$ is given by counting the solutions
$u$ of the Floer equation with $u(0,0)\in\zeta$. More precisely, pick
a generic almost complex structure and a generic $\zeta$ and set
$$
\Phi_{\zeta}(\bx):=\sum_{\by} m(\bx,\by;\zeta)\by.
$$
Here $m(\bx,\by;\zeta)$ is the number of the elements, taken with
appropriate signs, in the moduli space $\CM(\bx,\by;\zeta)$ of
solutions $u$ asymptotic to $\bx$ at $-\infty$ and
$\by$ at $\infty$ and such that $u(0,0)\in \zeta$. This moduli space
has dimension $|\bx|-|\by|-\codim(\zeta)$ and, by definition,
$m(\bx,\by;\zeta)=0$ when $\dim \CM(\bx,\by;\zeta)>0$.

The map $\Phi_\zeta$ commutes with the Floer differential and gives
rise to a well-defined map
$$
\Phi_{[\zeta]} \colon \HF^{(a,\,b)}_*(H)\to
\HF^{(a,\,b)}_{*-\codim(\zeta)}(H).
$$
The analytical details of this construction and complete proofs can be
found in, e.g., \cite{LO}, in much greater generality than is needed
here. Clearly,
$$
\Phi_{[M]}=\id.
$$

The action of the class $\alpha=\qq^\ell[\zeta]\in \HQ_*(M)$ is
induced by the map
$$
\Phi_{\qq^\ell\zeta}(\bx):=\sum_{\by} m(\qq^\ell\bx,\by;\zeta)\by.
$$
Here, as in Section \ref{sec:FFH}, $\qq=e^{A_0}$ where $A_0$ is the
generator of $\Gamma$ with $I_{c_1}(A_0)=-2N$. It is routine to check
that $\Phi_{\qq^\ell[\zeta]}=\qq^\ell\Phi_{[\zeta]}$. (This is a
consequence of the fact that
$ \CM(\qq^\ell\bx,\by;\zeta)=\CM(\bx,\qq^{-\ell}\by;\zeta)$.)  The
resulting map shifts the action interval by $I_\omega(\alpha)$, i.e.,
\begin{equation}
\label{eq:cap-action}
\Phi_\alpha \colon \HF^{(a,\,b)}_*(H)\to 
\HF^{(a,\,b)+I_\omega(\alpha)}_{*+|\alpha|-2n}(H),
\end{equation}
where $(a,\,b)+c$ stands for $(a+c,\,b+c)$.

By linearity over $\Lambda$, we extend $\Phi_\alpha$ to all
$\alpha\in \HQ_*(M)$ so that \eqref{eq:cap-action} still holds. The
maps $\Phi_\alpha$ are linear in $\alpha$ once the shift of the action
filtration is taken into account; see \cite[Sect.\
2.3]{GG:hyperbolic}.  These maps fit together to form an action of the
quantum homology on the collection of the filtered Floer homology
groups. The action is multiplicative. In other words, we have
\begin{equation}
\label{eq:qh-fh-action}
\Phi_\alpha\Phi_\beta=\Phi_{\alpha*\beta},
\end{equation}
which can be thought of as a form of associativity of the quantum
product.  Strictly speaking, in \eqref{eq:qh-fh-action}, as in the
case of additivity, the maps on the two sides of the identity have
different target spaces, which can be accounted for by considering the
shifts of the action filtration. We refer the reader to \cite[Sect.\
2.3]{GG:hyperbolic} for the precise statement. The identity
\eqref{eq:qh-fh-action} was essentially established in \cite{LO} and
\cite{PSS}; see also \cite[Rmk.\ 12.3.3]{MS}.

The cap product action also extends to the local Floer
homology. Namely, let $\bx$ be an isolated one-periodic orbit of $H$ and
let $\tH$ be a $C^2$-small non-degenerate perturbation of
$H$. Applying the above construction word-for-word to the complex
$\CF_*(\tH,x)$ from Section \ref{sec:LFH}, we obtain an action (i.e.,
a module structure) of $\HQ_*(M)$ on $\HF_*(\bx)$. In other words, for
every $\alpha\in \HQ_*(M)$ we have a map
$\Phi_\alpha\colon \HF_*(\bx)\to\HF_{*+|\alpha|-2n}(\bx)$ and
\eqref{eq:qh-fh-action} is satisfied. However, the resulting action is
trivial unless of course $\alpha$ is a multiple of $[M]$:

\begin{Lemma}
\label{lemma:trivial}
Assume that $|\alpha|<2n$. Then $\Phi_\alpha=0$ on $\HF_*(\bx)$.
\end{Lemma}

The proof of the lemma is standard and we just outline the argument.

\begin{proof}
  We argue as in the proof of the fact that $\CF_*(\tH,\bx)$ is a
  subcomplex in $\CF_*(\tH)$; cf.\ \cite{Gi:CC, GG:gaps}.  Let $B$ be
  a small ball centered at the fixed point $x(0)$ such that its
  closure $\bar{B}$ contains no other fixed points, and let $\tH$ be a
  $C^2$-small, non-degenerate perturbation of $H$. Denote by $\bx_i$
  the capped one-periodic orbits which $\bx$ splits into. Clearly,
  $x_i(0)\in B$.  Every class $\alpha$ with $|\alpha|<2n$ can be
  represented by a cycle avoiding $B$. Thus it is enough to show that
  for every solution $u$ of the Floer equation for $\tH$ asymptotic to
  some orbits $\bx_i$ and $\bx_j$, we have $u(s,0)\in B$ for all
  $s\in \R$ when $\tH$ is close to $H$. Arguing by contradiction
  assume that there is a sequence of $C^2$-small, non-degenerate
  perturbations $\tH_l\to H$ and solutions $u_l$ of the Floer equation
  for $\tH_l$ such that $u_l(s_l,0)\in \p B=\bar{B}\setminus
  B$. Without loss of generality we can assume that $s_l=0$ by
  applying a shift in $s$. Clearly, $E(u_l)\to 0$ and, as a
  consequence, the loops $t\mapsto u_l(0,t)$ converge to an integral
  curve of $\varphi_H^t$; see, e.g., \cite[Sect.\ 1.5]{Sa}. Passing to
  a subsequence, we conclude that $p=\lim u_l(0,0)$ is a fixed point
  of $\varphi_H$ which contradicts our choice of $B$.
\end{proof}

\begin{Remark}
  Lemma \ref{lemma:trivial} has a counterpart in the context of the
  algebra structure and the pair-of-pants product in Floer
  homology. This is the fact established in \cite{Ci} that, unless $x$
  is a so-called SDM, the local Floer homology algebra
  $\bigoplus_{k>0}\HF_*(\bx^k)$ is non-uniformly nilpotent.
\end{Remark}

\subsection{Spectral invariants and action carriers}
\label{sec:spec}
In this section, we briefly discuss spectral invariants and action
carriers to the extent necessary for our purposes. The theory of
Hamiltonian \emph{spectral invariants} was developed in its present
Floer theoretical form in \cite{Oh:constr, Sc}, although the first
versions of the theory go back to \cite{HZ, Vi:gen}. Action carriers
were introduced in \cite{GG:gaps} and then studied in \cite{CGG,
  GG:nm}.

Let $H$ be a Hamiltonian on a closed monotone (or even rational)
symplectic manifold $M^{2n}$. The \emph{spectral invariant} or
\emph{action selector} $\s_\alpha$ associated with a class
$\alpha\in \HF_*(H)=\HQ_*(M)$ is defined as
$$
\s_\alpha(H)= \inf\{ a\in \R\setminus \CS(H)\mid \alpha \in \im(i^a)\}
=\inf\{ a\in \R\setminus \CS(H)\mid j^a\left( \alpha \right)=0\},
$$
where $i^a\colon \HF_*^{(-\infty,\,a)}(H)\to \HF_*(H)$ and
$j^a\colon \HF_*(H)\to\HF_*^{(a,\, \infty)}(H)$ are the natural
``inclusion'' and ``quotient'' maps. It is easy to see that
$\s_\alpha(H)>-\infty$ when $\alpha\neq 0$ and one can show that
$\s_\alpha(H)\in S(H)$. In other words, there exists a capped
one-periodic orbit $\bx$ of $H$ such that $\s_\alpha(H)=\CA_H(\bx)$. As an
immediate consequence of the definition,
$$
\s_\alpha(H+a(t))=\s_\alpha(H)+\int_{S^1}a(t)\,dt,
$$
where $a\colon S^1\to\R$.

Action selectors have several important properties.  The function
$\s_\alpha$ is homotopy invariant: $\s_\alpha(H)=\s_\alpha(K)$ when
$\varphi_H=\varphi_K$ in $\tHam(M)$ and $H$ and $K$ have the same mean
value. Furthermore, it is sub-additive:
$$
\s_{\alpha * \beta}(H \nat K)\leq\s_{\alpha}(H)+\s_{\beta}(K).
$$
In particular,
$$
\s_{[M]}(H \nat K)\leq\s_{[M]}(H)+\s_{[M]}(K).
$$
Finally, $\s_\alpha$ is monotone and Lipschitz in the $C^0$-topology
as a function of $H$.

When $H$ is non-degenerate, the action selector can also be evaluated
as
$$
\s_\alpha(H)=\inf_{[\sigma]=\alpha}\s_\sigma(H),
$$
where we set 
\begin{equation}
\label{eq:cycle-action}
\s_\sigma(H)=\max\big\{\CA_H(\bx)\,\big|\, \sigma_{\bx} 
\neq 0\big\}\text{ for }
\sigma=\sum\sigma_{\bx} \bx\in\CF_*(H).
\end{equation}
The infimum here is attained, since $M$ is rational and thus $\CS(H)$
is closed.  Hence there exists a cycle
$\sigma=\sum\sigma_{\bx} \bx\in\CF_{|\alpha|}(H)$, representing the
class $\alpha$, such that $\s_\alpha(H)=\CA_H(\bx)$ for an orbit $\bx$
entering $\sigma$. In other words, $\bx$ maximizes the action on
$\sigma$ and the cycle $\sigma$ minimizes the action over all cycles
in the homology class $\alpha$. Such an orbit $\bx$ is called a
\emph{carrier} of the action selector. This is a stronger requirement
than just that $\s_\alpha(H)=\CA_H(\bx)$ and $\mu(\bx)=|\alpha|$. When
$H$ is possibly degenerate, a capped one-periodic orbit $\bx$ of $H$
is a carrier of the action selector if there exists a sequence of
$C^2$-small, non-degenerate perturbations $\tH_i\to H$ such that one
of the capped orbits $\bx$ splits into is a carrier for $\tH_i$. An
orbit (without capping) is said to be a carrier if it turns into one
for a suitable choice of capping.

It is easy to see that a carrier necessarily exists, but, in
general, is not unique. However, it becomes unique when all one-periodic
orbits of $H$ have distinct action values.

As consequence of the definition of the carrier and continuity of the
action and the mean index, we have
$$
\s_\alpha(H)
=\CA_H(\bx)\text{ and } \big|\hmu(\bx)-|\alpha|\big| \leq n,
$$
and the inequality is strict when $x$ is weakly non-degenerate.
Furthermore, a carrier $\bx$ for $\s_\alpha$ is in some sense
homologically essential as the following result asserts.

\begin{Lemma}
\label{lemma:ac}
Let $\bx$ be an action carrier for $\s_\alpha$. Then
$\HF_{|\alpha|}(\bx)\neq 0$ when $x$ is isolated.
\end{Lemma}

This lemma is proved in \cite[Lemma 3.2]{GG:nm} for $\s_{[M]}$, but
the proof goes through word-for-word for any homology class. For the
sake of brevity, we omit the argument.

\section{Background results on pseudo-rotations}
\label{sec:background}
In this section, we assemble several known symplectic topological
results on pseudo-rotations of projective spaces, crucial for what
follows.

Let $\varphi=\varphi_H$ be a pseudo-rotation of $\CP^n$, which we do
not assume to be non-degenerate. Denote by $\alpha_l$ the generator in
$\HF_{2l-n}(H)=\F$, $l\in\Z$, and let $\bx_{l}\in\bPP_1(H)$ be an action
carrier for $\alpha_l$. We will write $\s_l:=\s_{\alpha_l}$. Then, in
particular,
$$
\s_{l}(H)=\CA_H(\bx_l)\textrm{ and } \HF_{2l-n}(\bx_l)\neq 0.
$$
Throughout the paper it is convenient to occasionally rescale the
standard Fubini--Study symplectic structure $\omega$ on
$\CP^n$. Hence, we state the results in a slightly more general form
assuming that $\omega$ is proportional to the Fubini--Study form and
$\lambda=\left<\omega,\CP^1\right>$ is the rationality constant. (For
the Fubini--Study form, $\lambda=\pi$.)

\begin{Theorem}[\cite{GG:gaps}]
\label{thm:bijection}
For every $l\in \Z$ the action carrier $\bx_l$ is unique and the
resulting map
\begin{equation}
\label{eq:map-x_i}
\Z\to\bPP_1(H), \quad l\mapsto\bx_l
\end{equation}
is a bijection. Furthermore, the map
\begin{equation}
\label{eq:map-s}
\Z\to\CS(H), \quad l\mapsto \s_{l}(H)=\CA_H(\bx_l)
\end{equation}
is strictly monotone, i.e., $l>l'$ if and only if
$\CA_H(\bx_l)>\CA_H(\bx_{l'})$.  
\end{Theorem}

One important consequence of the theorem is that distinct capped
one-periodic orbits of $\varphi_H$ necessarily have different
actions. Another is that, by Lemma \ref{lemma:ac},
$\HF_{2l-n}(\bx_l)\neq 0$. In particular, $\HF(x)\neq 0$ for all
$x\in\PP(H)$.

When $H$ is non-degenerate we have $\mu(\bx_l)=2l-n$, and the proof of
the theorem is rather straightforward. The general case of the theorem
is established in \cite[Sect.\ 6]{GG:gaps}. First, one shows that the
map \eqref{eq:map-s} is strictly monotone. This is a consequence of
the facts that the orbits are isolated and
$[\CP^{n-1}] * \alpha_l=\alpha_{l-1}$; see \cite[Prop.\
6.2]{GG:gaps}. Next, we have the relation
\begin{equation}
\label{eq:lambda-bound}
\s_{l+(n+1)}(H)=\s_l(H)+\lambda,
\end{equation}
which follows from \eqref{eq:qp-cpn}.  Now the chain of the
inequalities
$$
\cdots<\s_0(H)<\dots<\s_n(H)<\s_{n+1}(H)=\s_0(H)+\lambda<\cdots,
$$
is used to infer the uniqueness of $\bx_l$ and the assertion that the
map \eqref{eq:map-x_i} is a bijection. (The argument is similar to the
proof of the degenerate case of Arnold's conjecture for $\CP^n$; see
\cite{Fl, Fo, FW} and also \cite{Oh:constr, Sc:invent} and a detailed
account in \cite[Sect.\ 6.2.3]{GG:gaps}.) Note also that we
automatically have the identity $\bx_{l+ (n+1)}=\bx_l\# \CP^1$.

\begin{Example}[Rotations of $\CP^n$, I]
\label{ex:cpn}
In the context of this paper, a rotation of $\CP^n$ is a Hamiltonian
diffeomorphism $\varphi_Q$ generated by a quadratic Hamiltonian
$Q=\sum a_i|z_i|^2$, where we identified $\CP^n$ with the quotient of
the unit sphere $S^{2n+1}\subset \C^{n+1}$ by the diagonal (Hopf)
$S^1$-action. A calculation shows that $\varphi_Q$ is strongly
non-degenerate if and only if it is a pseudo-rotation and if and only
if $a_j-a_i\not\in\Q$ for all pairs $j\neq i$. Then
$\PP(Q)=\PP_1(Q)=\{x_0,\ldots,x_n\}$ is the set of the coordinate axes
and, without loss of generality, we may assume that
$a_0<\ldots<a_n$. Furthermore, we can also require that $\sum a_i=0$
or, equivalently, that the mean value of $Q$ is zero; for the
Hamiltonian $\sum|z_i|^2$ is constant on $\CP^n$. We have
$$
\CS(H)=\bigsqcup a_i+\lambda\Z,
$$
where $\lambda=\pi$ when $\omega$ is the standard Fubini--Study
symplectic form. Denote by $\hx_i$ the constant orbit $x_i$ equipped
with the trivial capping. By a calculation, one can see that the
eigenvalues of $d\varphi_Q$ at $x_i$ are
$\exp\big(\pm 2\lambda \sqrt{-1}(a_j-a_i)\big)$, where $j\neq i$, and
$\hmu(\hx_i)=2(n+1)a_i/\lambda$. If $Q$ is $C^2$-small and non-degenerate,
$\mu(\hx_i)=2i-n$. We will elaborate on this example in \cite{GG:PR2}.
\end{Example}

Theorem \ref{thm:bijection} enables us to extend the notion of the
Conley--Zehnder index to capped one-periodic orbits $\bx_l$ of
degenerate pseudo-rotations by setting $\mu(\bx)=2l-n$ for
$\bx=\bx_l$. We will call $\mu(\bx)$ the \emph{LS-index}; for it
ultimately comes from a version of Lusternik--Schnirelmann theory for
action selectors which the proof of the theorem is based on. When $x$,
the one-periodic orbit underlying $\bx$, is non-degenerate this is
just the ordinary Conley--Zehnder index. Without non-degeneracy, the
LS-index is a global rather than local invariant. However, it has some
of the expected properties of the Conley--Zehnder index, e.g.,
$\mu(\bx\#A)=\mu(\bx)+I_{c_1}(A)$ and $\HF_{\mu(\bx)}(\bx)\neq
0$. Moreover, when $n=1$, $\HF(\bx)$ is concentrated in only one
degree which is $\mu(\bx)$; see \cite{GG:gap}. With this notion in
mind, Theorem \ref{thm:bijection} can be rephrased as that the
ordering of $\bPP_1(H)$ by the LS-index agrees with that by the
action.

Next, recall that the \emph{augmented action} of $H$ on $x\in\PP_1(H)$
is defined by
\begin{equation}
\label{eq:aug-action}
\tCA_H(x)=\CA_H(\bx)-\frac{\lambda}{2N}\hmu(\bx),
\end{equation}
where on the right we used an arbitrary capping of $x$. (Note that the
augmented action is defined for every monotone manifold; see
\cite{GG:gaps}.) It is clear that the augmented action is well
defined, i.e., independent of the capping, and homogeneous with
respect to iteration. Furthermore, $\tCA_H(x)$ is completely
determined, once $H$ is normalized to have zero mean, by the time-one
map $\varphi_H$ viewed as an element of $\Ham(\CP^n)$ rather than
$\tHam(\CP^n)$; \cite{EP}.

\begin{Theorem}[Action-index Resonance Relations; Thm.\ 1.12 in
  \cite{GG:gaps}]
\label{thm:act-index}
Let $x_0,\ldots, x_n$ be the fixed points of a pseudo-rotation
$\varphi_H$ of $\CP^n$. Then
$$
\tCA_H(x_0)=\ldots=\tCA_H(x_n).
$$
\end{Theorem}

This theorem has been extended to some other symplectic manifolds and
a broader class of Hamiltonian diffeomorphisms; see \cite{CGG}. It
also has an analog for Reeb flows; \cite[Sect.\ 6.1.2]{GG:convex}.

As a consequence of Theorem \ref{thm:act-index},
\begin{equation}
\label{eq:act-index2}
\CA_H(\bx)=\frac{\lambda}{2(n+1)}\hmu(\bx)+\const
\end{equation}
for every $\bx\in\bPP_1(H)$, where $\const=\tCA_H(x)$ is independent of
$x$. Surprisingly, very little is known about
$\tCA_H(x)$. Conjecturally, $\tCA_H(x)=0$ when $H$ is normalized to
have zero mean. It is known however (see \cite[Thm.\ 5.1]{GG:Rev})
that
\begin{equation}
\label{eq:quasi}
\tCA_H(x)=\lim_{k\to\infty}\frac{\s_{[M]}\big(H^{\nat
    k}\big)}{k}
\end{equation}
for pseudo-rotations of $M=\CP^n$. The limit on the right is the
asymptotic spectral invariant, extensively studied in \cite{EP} where
it is shown to be equal to a Calabi quasimorphism for many symplectic
manifolds including $\CP^n$ once $H$ is normalized to have zero mean.

The \emph{marked action spectrum} $\cCS(H)$ is, by definition, the
bijection
$$
\cCS\colon
\Z\stackrel{\eqref{eq:map-x_i}}{\longleftrightarrow}\bPP_1(H)
\stackrel{\eqref{eq:map-s}}{\longrightarrow}\CS(H),
$$
i.e., $\cCS(H)$ is simply the spectrum $\CS(H)$ with its points
labelled by $\Z$ (essentially the indices) or, equivalently, by
$\bPP_1(H)$. In a similar vein, the \emph{marked index spectrum}
$\cCSI(\varphi_H)$ is the map
$$
\cCSI\colon
\Z\stackrel{\eqref{eq:map-x_i}}{\longleftrightarrow}\bPP_1(H)
\longrightarrow\CSI(\varphi),
$$
where 
$$
\CSI(\varphi)=\{\hmu(\bx)\mid x\in\bPP_1(H)\}
$$ 
is the \emph{mean index spectrum} of $H$ and the second arrow is the
map $\bx\mapsto \hmu(\bx)$.  Then \eqref{eq:act-index2} can be
rephrased as
$$
\cCS(H)=\frac{\lambda}{2(n+1)}\cCSI(\varphi_H)+\const,
$$
i.e., the action spectrum and the index spectrum agree up
to a factor and a shift. The factor can be made equal 1 by scaling
$\omega$, and the shift can be made zero by adding a constant to $H$:
\begin{equation}
\label{eq:equal-spectra}
\cCS(H)=\cCSI(\varphi_H).
\end{equation}

Finally, we have a different type of resonance relations involving
only the indices. To state the result, recall that for an uncapped
one-periodic orbit $x\in\PP_1(H)$ the mean index is well defined
modulo $2(n+1)=2N$, i.e., $\hmu(x)\in S_{2N}^1=\R/2(n+1)\Z$.

\begin{Theorem}[Mean Index Resonance Relations; Sect.\ 1.2 in
  \cite{GK}]
\label{thm:index-RR}
Let $x_0,\ldots, x_n$ be the fixed points of a pseudo-rotation
$\varphi_H$ of $\CP^n$. Then for some non-zero vector
$\vr=(r_0,\ldots,r_n)\in \Z^{n+1}$, we have
$$
\sum r_i\hmu(x_i)= 0 \textrm{ in } \R/2(n+1)\Z.
$$
\end{Theorem}

In other words, the closed subgroup $\Gamma\subset \T^{n+1}$
topologically generated by the \emph{mean index vector}
\begin{equation}
\label{eq:vDelta}
\vDelta=\vDelta(\varphi_H):=\big(\hmu(x_0),\ldots,\hmu(x_n)\big)\in
\T^{n+1}=\R^{n+1}/2(n+1)\Z^{n+1}
\end{equation}
has positive codimension. Moreover, the codimension is equal to the
number of linearly independent resonances, i.e., the rank of the
subgroup $\CR\subset\Z^{n+1}$ formed by all resonances $\vr$; see
\cite{GK}.

Clearly, $\CR$ depends on $\varphi_H$. However, conjecturally, the
resonance relation
\begin{equation}
\label{eq:sum-RR}
\sum\hmu(x_i)=0 \textrm{ in } \R/2(n+1)\Z.
\end{equation}
is universal, i.e., satisfied for all pseudo-rotations. (Up to a
factor this is the only possible universal resonance relation; for, as
is easy to see, any other relation breaks down for a suitably chosen
rotation; \cite{GG:PR2}.) For $S^2$, \eqref{eq:sum-RR} asserts,
roughly speaking, that $D\varphi$ rotates the tangent spaces at the
fixed points by the same angle but in opposite directions. This is a
consequence of the Poincar\'e--Birkhoff theorem, \cite[Appendix
A.2]{Br:Annals}; see also \cite{CKRTZ} for a different approach based
on Theorem \ref{thm:index-RR}. In \cite{GG:PR2} we will revisit this
conjecture and establish it for non-degenerate pseudo-rotations of
$\CP^2$.

\begin{Example}[Rotations of $\CP^n$, II]
\label{ex:cpn-2}
In the setting of Example \ref{ex:cpn}, we have $\tCA_Q(x_i)=0$ as a
direct calculation shows. As we have done throughout this section, let
us denote by $\bx_i$ the orbit $x_i$ capped so that
$\mu(\bx_i)=2i-n$. (By Theorem \ref{thm:bijection}, such a capping
exists and is unique.) When $Q$ is $C^2$-small, $\bx_i=\hx_i$. One can
show that
$$
\sum\hmu(\bx_i)=0,
$$
although this is not obvious. Thus, in particular, \eqref{eq:sum-RR}
holds. When $Q$ is $C^2$-small, this follows from that $\sum a_i=0$
and $\hmu(\hx_i)=2(n+1)a_i/\lambda$ by Example \ref{ex:cpn}.  The
general case is proved in \cite{GG:PR2}.
\end{Example}

We conclude this discussion by the following sample illustrative
application of the three theorems from this section to dynamics of
pseudo-rotations.

\begin{Corollary}
\label{cor:S^2-nondeg}
Every pseudo-rotation of $S^2$ is strongly non-degenerate and its
fixed points are elliptic.
\end{Corollary}

The corollary is a standard, although ultimately highly non-trivial,
result in two-dimensional dynamics; see \cite{Fr92}.

\begin{proof} Arguing by contradiction, assume that $\varphi_H$ is a
  pseudo-rotation of $S^2$ and some iterate of $\varphi_H$ is
  degenerate or that one of its fixed points is hyperbolic. In
  dimension two, a hyperbolic or degenerate fixed point necessarily
  has integer mean index. Hence, replacing $\varphi_H$ by a
  sufficiently large iterate and using Theorem \ref{thm:index-RR}, we
  may assume that both fixed points $x$ and $y$ of $\varphi_H$ have
  mean index equal to zero modulo 4. Let us scale the symplectic
  structure and adjust the Hamiltonian so that
  $\cCS(H)=\cCSI(\varphi_H)$. Then, by \eqref{eq:act-index2},
  $\CA_H(\bx)=\CA_H(\by)$ for suitable cappings of $x$ and $y$, which
  is impossible by Theorem \ref{thm:bijection}.
  \end{proof}

\section{Invariant sets}
\label{sec:sets}

Our goal in this section is to partially generalize some of the
results of Le Calvez and Yoccoz, \cite{LCY}, and of Franks and
Misiurewicz, \cite{Fr99,FM}, on invariant sets in dimension two to
higher dimensions. To lay out the context, recall that the key feature
of the examples of pseudo-rotations of $S^2$ constructed by Anosov and
Katok in \cite{AK} (see also, e.g., \cite{FK}) is that these
pseudo-rotations $\varphi$ have exactly three ergodic measures: the
two fixed points and the area form. Thus on the complement to the two
fixed points $\varphi$ is uniquely ergodic and even though $\varphi$
is not uniquely ergodic on $S^2$ it is very close to being so. Volume
preserving, uniquely ergodic maps of compact manifolds are necessarily
minimal, i.e., all orbits are dense; see, e.g., \cite{Wa}. Therefore,
one might also expect every orbit of $\varphi$ other than a fixed
point to be dense. (In a similar vein, every symplectic manifold that
admits a symplectic $S^1$-action without fixed points also admits a
minimal symplectomorphism; \cite{HCP}.)  This turns out to be
false. As shown in \cite[Prop.\ 5.5]{FM} drawing from the
aforementioned results, the set of periodic points of a topologically
transitive homeomorphism of $S^2$ cannot be simultaneously finite and
isolated as an invariant set. As a consequence, a pseudo-rotation of
$S^2$ must have orbits other than the fixed points entirely contained
in an arbitrarily small neighborhood of the fixed point set. While the
full scope of \cite[Prop.\ 5.5]{FM} is certainly out of the reach of
symplectic methods, its application to area-preserving smooth
pseudo-rotations fits well in the symplectic framework. We have the
following partial generalization of this result already mentioned in
the introduction as Theorem \ref{thm:inv_sets0}.

\begin{Theorem}
\label{thm:inv_sets}
No fixed point of a pseudo-rotation of $\CP^n$ is isolated as an
invariant set.
\end{Theorem}

The theorem is in turn a consequence of the following general result;
cf.\ Example~\ref{ex:cpn}.

\begin{Theorem}
\label{thm:isolated}
Let $M^{2n}$ be a strictly monotone symplectic manifold. Assume that
the minimal Chern number $N\geq n+1$ and
$$
\alpha *\beta = \qq [M]
$$
in $\HQ_*(M)$ for some homology classes $\alpha\in \H_*(M)$ and
$\beta\in \H_*(M)$ with $|\alpha|<n$ and $|\beta|<n$.  Let $\varphi$
be a Hamiltonian diffeomorphism of $M$ with a contractible periodic
orbit $x$ which is isolated as an invariant set and such that
$\HF(x^k)\neq 0$ for all $k\in\N$. Then $\varphi$ has infinitely many
periodic points.
\end{Theorem}

Theorem \ref{thm:isolated}, proved in Section
\ref{sec:energy-isolated}, significantly relaxes the dynamical
requirements in \cite[Thm.\ 1.1]{GG:hyperbolic}, where the orbit $x$
is assumed to be hyperbolic, at the expense of imposing more
restrictive conditions on $M$ and, in particular, on $N$. As in that
theorem the assumption that the orbit $x$ is contractible can be
relaxed when $M$ is toroidally monotone, but we have no examples of
positive monotone manifolds with $\pi_1(M)\neq 1$ meeting the
conditions of the theorem. (See however Remark \ref{rmk:isolated}.)
In fact, the only positive monotone manifold with $N\geq n+1$ known to
us is $\CP^n$. The requirements of the theorem are met by numerous
negative monotone manifolds, but in this case the Conley conjecture
holds and $\varphi$ has infinitely many periodic orbits
unconditionally; see \cite{CGG, GG:nm}. However, the result we
actually prove (Theorem \ref{thm:isolated2}) is slightly more precise
than Theorem \ref{thm:isolated} and it gives additional information
about the augmented actions of periodic orbits, even when $M$ is
negative monotone.

We also note that the assumption that a periodic orbit $x$ is isolated
as an invariant set imposes a strong restriction on the dynamics. In
particular, $x$ can easily be isolated as a periodic orbit for all
iterations but not isolated as an invariant set.

\begin {proof}[Proof of Theorem \ref{thm:inv_sets}] 
  Let $\varphi$ be a Hamiltonian pseudo-rotation of $\CP^n$. As is
  pointed out in Section \ref{sec:background}, $\HF(x)\neq 0$ for
  every fixed point $x$ of $\varphi$ by Lemma \ref{lemma:ac}. Since
  $\varphi^k$ is also a pseudo-rotation for every $k\in\N$, we have
  $\HF(x^k)\neq 0$. Now it remains to apply Theorem \ref{thm:isolated}
  to $M=\CP^n$. Indeed, if $x$ were isolated as an invariant set,
  $\varphi$ would have infinitely many periodic orbits by Theorem
  \ref{thm:isolated}, and hence would not be a pseudo-rotation.
\end{proof}

\begin{Remark}
  The condition that $\HF(x^k)\neq 0$ is automatically satisfied when
  the Hopf index of $x^k$ is non-zero. In all examples known to us,
  $\HF(x^k)\neq 0$ for all $k\in\N$ whenever $\HF(x)\neq 0$ and $x^k$
  is isolated. Furthermore, $\HF(x^k)\neq 0$ when $\HF(x)\neq 0$ and
  $k$ is admissible, i.e., none of the Floquet multipliers of $x$ is a
  root of unity of degree $k$ (see \cite[Thm.\ 1.1]{GG:gap}), and we
  conjecture that this is true for all $k$.
  \end{Remark}

\section{Lagrangian Poincar\'e recurrence and $C^0$-rigidity}
\label{sec:PR+PO}
In this section we prove two results which roughly speaking assert
that under suitable extra conditions a sequence of iterates
$\varphi^{k_i}$ of a pseudo-rotation converges to the identity in the
appropriate metric. The first of these results (Theorem
\ref{thm:gamma}) is closely related to the Lagrangian Poincar\'e
recurrence conjecture (see Theorem \ref{thm:LPR}) while the second
(Theorem \ref{thm:exp-L}) is a higher-dimensional analog of the
$C^0$-rigidity theorem from~\cite{Br}.

\subsection{Lagrangian Poincar\'e recurrence and the $\gamma$-norm
  convergence}
\label{sec:PR}
\subsubsection{The $\gamma$-norm convergence}
\label{sec:gamma}
We start this section by briefly recalling the definition of the
$\gamma$-norm and the related $\gamma$-metric on $\Ham(M)$. This norm
was originally introduced in \cite{HZ, Vi:gen} for compactly supported
Hamiltonian diffeomorphisms of $\R^{2n}$ using generating
functions. The construction was then extended to closed symplectically
aspherical manifolds in \cite{Sc} and to all rational weakly monotone
manifolds in \cite{Oh:gamma}. We refer the reader to these papers for
the original definitions and a much more detailed treatment.

Let $M^{2n}$ be a closed, weakly monotone, rational symplectic
manifold and let $\varphi$ be a Hamiltonian diffeomorphism of $M$
generated by a Hamiltonian $H$. It is convenient to assume that
$H_t\equiv 0$ for $t$ close to $0$. Set
$$
H^{\inv}_t=-H_t\circ \varphi_t.
$$  
This is a periodic in time Hamiltonian generating the time-dependent
flow $\big(\varphi^t\big)^{-1}$. (Alternatively, one can take the
Hamiltonian generating the time-dependent flow $\varphi_H^{-t}$ as
$H^\inv$.) By definition,
$$
\gamma(H)=\s_{[M]}(H)+\s_{[M]}(H^\inv) \textrm{ and }
\gamma(\varphi)=\inf_H\gamma(H),
$$
where the infimum is taken over all $H$ generating the time-one map
$\varphi$. (In fact, $\gamma(H)$ is independent of $H$ as long as the
time-dependent flow $\varphi_H^t$ remains in the same homotopy class
with fixed end points. For a broad class of manifolds, but not for
$\CP^n$, we have $\gamma(H)=\gamma(\varphi)$ for all Hamiltonians $H$
generating $\varphi$.)  Furthermore, set
$$
d_\gamma(\varphi,\psi):=\gamma\big(\varphi\psi^{-1}\big).
$$
It is a standard fact that $d_\gamma$ is a right-invariant distance on
$\Ham(M)$; see, e.g., \cite{Oh:gamma}.

As a consequence of the Poincar\'e duality in Floer homology (see,
e.g., \cite[Lemma 2.2]{EP}), for some symplectic manifolds including
$\CP^n$, we have
\begin{equation}
\label{eq:H^inv}
\s_{[M]}(H^\inv)=-\s_{[\pt]}(H)
\end{equation}
and thus
\begin{equation}
\label{eq:gamma(H)}
\gamma(H)=\s_{[M]}(H) -\s_{[\pt]}(H).
\end{equation}

\begin{Theorem}
\label{thm:gamma}
Let $\varphi=\varphi_H$ be a pseudo-rotation of $\CP^n$. Then there
exist a constant $C>0$ and a non-negative integer $d\leq n$, both
depending only on $\vDelta(\varphi)$, such that for every $\eps$ such
that $0<\eps\leq \lambda$, we have
\begin{equation}
\label{eq:gamma}
\liminf_{k\to\infty} \frac{|\{\ell\leq k \mid 
\gamma(\varphi^\ell)<\eps\}|}{k}\geq C\eps^d.
\end{equation}
In particular, the limit inferior is positive.
\end{Theorem}

\begin{Corollary}[$\gamma$-norm convergence]
\label{cor:gamma}
Let $\varphi$ be a pseudo-rotation of $\CP^n$. Then
$\gamma(\varphi^{k_i})\to 0$ for some sequence $k_i\to\infty$.
\end{Corollary}

\begin{Remark}
\label{rmk:gamma}
To the best of our knowledge, both the theorem and the corollary are
new even when $n=1$.  It might be possible to relax the conditions of
Theorem \ref{thm:gamma} by combining its proof with the proof of
\cite[Thm.\ 1.1]{CGG} and extend the theorem, or at least Corollary
\ref{cor:gamma}, to perfect Hamiltonian diffeomorphisms of $\CP^n$.
Furthermore, recall that $\gamma(\varphi)$ is \emph{a priori} bounded
for all $\varphi\in\Ham(\CP^n)$ and, in fact,
$\gamma(\varphi)\leq \lambda$; see \cite{EP} and also \cite{McD}. A
simple way to see this in the context of the paper is to use
\eqref{eq:H^inv} and \eqref{eq:gamma(H)} together with the fact that
$\s_{[\pt]}(H)\geq \s_{[M]}(H)-\lambda$ by
\eqref{eq:lambda-bound}. This shows that the condition that
$\eps\leq \lambda$ is not really restrictive: when $\eps>\lambda$ the
density on the left-hand side of \eqref{eq:gamma} is equal to one. In
fact, the theorem is most interesting for small values of $\eps$.

  The constant $d$ in Theorem \ref{thm:gamma} and also Theorem
  \ref{thm:LPR} below is the difference $d=n+1-r$, where $r$ is the
  number of linearly independent resonance relations the mean indices
  $\hmu(x_i)$ satisfy; see \cite{GK} or Section \ref{sec:background}.
\end{Remark}

\begin{Remark}
  Few manifolds $M$ are expected to admit Hamiltonian diffeomorphisms
  $\varphi$ such that $\gamma\big(\varphi^{k_i}\big)\to 0$ for some
  sequence $k_i\to\infty$. For instance, hypothetically, this is never
  the case when $M$ is symplectically aspherical. (As far as we know,
  this question/conjecture is due to L. Polterovich;
  we learned of it from Seyfaddini.) There is also a similar
  question for the $C^0$- or $C^1$-norm and in this instance some partial
  results are available. For example, it is not hard to see that one
  can never have $\varphi^{k_i}\to \id$ in the $C^1$-sense when $M$ is
  symplectically aspherical; cf.\ \cite{Po}.
\end{Remark}

\begin{proof}
  Throughout the proof, it will be convenient to rescale the
  symplectic structure on $\CP^n$ so that $[\omega]=2c_1(T\CP^n)$ and
  hence $\lambda=2(n+1)$, and to normalize the Hamiltonian to ensure
  that all fixed points have zero augmented action. Then
  \eqref{eq:equal-spectra} holds, $\cCS(H)=\cCSI(\varphi)$, or, in
  other words,
$$
\CA_H(\bx)=\hmu(\bx)
$$
for every capped one-periodic orbit $\bx$ of $\varphi$. Since the
augmented action is homogeneous this is also true for all iterates
$H^{\nat k}$. Furthermore, since
$\gamma(\varphi^k)\leq \gamma\big(H^{\nat k}\big)$, it suffices to
prove the theorem for $\gamma\big(H^{\nat k}\big)$ in place of
$\gamma(\varphi^k)$.

Observe also that we only need to prove the theorem for small
$\eps>0$, e.g., for every $\eps< 4$. Then the result for the entire range of
$\eps$ from $0$ to $\lambda=2(n+1)$ will follow by adjusting the value
of $C$.

Consider the mean index vector $\vDelta$ defined by \eqref{eq:vDelta}
and let $\Gamma\subset \T^{n+1}$ be the subgroup topologically
generated by this vector. Thus $\Gamma$ is the closure of the positive
semi-orbit $\CO=\{k\vDelta\mid k\in\N\}$. The connected component of
the identity in $\Gamma$ is a torus, and $\Gamma$ is isomorphic to the
direct product of this torus and a cyclic group of order
$k_0$. Replacing $\varphi$ by $\varphi^{k_0}$ we can assume that
$\Gamma$ is connected, and hence isomorphic to a torus. Set
$d=\dim\Gamma$. Since the components of $\vDelta$ satisfy at least one
resonance relation by Theorem \ref{thm:index-RR}, we have $d\leq n$.

The volume of the intersection of the $\eps/2$-neighborhood $B(\eps)$
of $0$ in $\T^{n+1}$ with $\Gamma$ is bounded from below by $C\eps^d$,
where $C$ is determined by the geometry of $\Gamma$:
\begin{equation}
\label{eq:vol}
\vol\big(B(\eps)\cap\Gamma\big)\geq C\eps^d.
\end{equation}
The semi-orbit $\CO$ is uniformly distributed in $\Gamma$ and hence to
prove \eqref{eq:gamma} it is enough to show that
\begin{equation}
\label{eq:eps}
\gamma\big(H^{\nat k}\big)<\eps\textrm{ whenever }
k\vDelta\in B(\eps).
\end{equation}

To this end, it is convenient to equip $\T^{n+1}$ with the metric
generated by the norm (the distance to the origin)
$$
\|\vtheta\|=\max_i\|\theta_i\|,
$$
where $\vtheta=(\theta_0,\ldots,\theta_n)\in\T^{n+1}$ and $\|\cdot\|$
on the right-hand side stands for the distance to zero in
$\R/2(n+1)\Z$. Thus $B(\eps)$ is actually a cube with faces
perpendicular to the ``coordinate axes'' and the diameter of
$\T^{n+1}$ is $2(n+1)$. (The choice of a norm on $\T^{n+1}$ effects
only the value of the constant $C$ in \eqref{eq:vol} and
\eqref{eq:gamma}.)

Assume that $k\vDelta\in B(\eps)$. Then the spectrum
$\CS\big(H^{\nat k}\big)=\CSI(\varphi^k)$ is contained in the
$\eps/2$-neighborhood of $2(n+1)\Z$. We will call the part of this
spectrum lying in the $\eps/2$-neighborhood of one point of $2(n+1)\Z$
a cluster. To finish the proof we simply need to show that
$\s_{[M]}\big(H^{\nat k}\big)$, where $M=\CP^n$, and
$\s_{[\pt]}\big(H^{\nat k}\big)$ are in the same cluster. Indeed, then
$$
\gamma\big(H^{\nat k}\big)=\s_{[M]}\big(H^{\nat
  k}\big)-\s_{[\pt]}\big(H^{\nat k}\big)<\eps.
$$
We will prove that these action values are in fact in the cluster
centered at $0$. (Up to this point we could have worked directly with
the ``action vector'' instead of $\vDelta$, an element of $\T^{n+1}$
whose components are the actions of uncapped one-periodic orbits
viewed as points in $\R/\lambda\Z$. However, in the next step, the role
of $\vDelta$ becomes essential because of \eqref{eq:supp}.)

Focusing on $\s_{[\pt]}\big(H^{\nat k}\big)$, denote by $2(n+1)q$ the
center of the cluster containing this point. Our goal is to show that
$q=0$. Let $\bx$ be the action carrier for $[\pt]$, i.e., $\bx$ is the
capped $k$-periodic orbit uniquely determined by the condition
$$
\s_{[\pt]}\big(H^{\nat k}\big)=\CA_{H^{\nat k}}(\bx).
$$
This orbit has LS-index $-n$. (Recall that in the non-degenerate case
this is simply the Conley--Zehnder index.) Since
$$
\mu(\bx)=-n\in
\supp \HF(\bx)\subset [\hmu(\bx)-n,\,\hmu(\bx)+n]
$$
by \eqref{eq:supp}, we have $-2n\leq \hmu(\bx)\leq 0$. On the other
hand, $\big|\hmu(\bx)-2(n+1)q\big|<\eps/2$, and thus $q=0$ due to the
assumption that $\eps<4$. A similar argument shows that
$\s_{[M]}\big(H^{\nat k}\big)$ is also in the cluster centered at
$0$. This proves \eqref{eq:eps} and completes the proof of the
theorem.
\end{proof}

\begin{Remark}
\label{rmk:Delta-gamma}
It is clear from the proof that we have also established the inequality
$$
\gamma\big(\varphi^k\big)\leq \const\, \|\vDelta\big(\varphi^k\big)\|,
$$
where $\const>0$ depends only on $\vDelta(\varphi)$. In particular, if
$\|\vDelta\big(\varphi^{k_i}\big)\|$ converges to zero, the sequence
$\gamma\big(\varphi^{k_i}\big)$ converges to zero at least as fast.
\end{Remark}

\subsubsection{Lagrangian Poincar\'e recurrence}
Consider a compactly supported Hamiltonian diffeomorphism $\varphi$ of
a symplectic manifold $M^{2n}$. The following conjecture was put forth
by the first author and independently by Claude Viterbo around 2010.

\begin{Conjecture}[Lagrangian Poincar\'e Recurrence]
\label{conj:LPR}
For any closed Lagrangian submanifold $L\subset M$ there exists a
sequence of iterations $k_i\to\infty$ such that
$$\varphi^{k_i}(L)\cap L\neq \emptyset.$$ 
Moreover, the density of the sequence $k_i$ is related to a symplectic
capacity of $L$.
\end{Conjecture}

The requirement that $\varphi$ is Hamiltonian is essential: the
conjecture fails for symplectomorphisms of $\T^2$, e.g., for an
irrational shift. Furthermore, the conjecture is most interesting when
$L$ is ``small''. When it is not, e.g., if $L$ is not displaceable,
the assertion is often obvious.  In dimension two (i.e., for $n=1$),
the conjecture readily follows from the standard Poincar\'e recurrence
theorem when $L$ bounds and the observation that otherwise $L$ is not
displaceable by a Hamiltonian diffeomorphism. On the other hand, to
the best of our knowledge, beyond $n=1$ the question has been
completely open prior to now. Surprisingly, the conjecture is not
straightforward to prove even for a given Hamiltonian diffeomorphism
$\varphi$ unless, of course, it is periodic, i.e., $\varphi^k=\id$ for
some $k$. The difficulty is present already when $\varphi$ is as
simple as an irrational rotation of $\CP^n$, i.e., the Hamiltonian
diffeomorphisms generated by a quadratic Hamiltonian.

It is also worth pointing out that it is sufficient to prove the
existence of one iteration $k=k(\varphi)>1$, the first return time,
such that $\varphi^k(L)\cap L\neq 0$ for every $\varphi$ or at least
for the iterates of a fixed map. Then the existence of infinitely many
such iterates will follow by replacing $\varphi$ by $\varphi^k$ and
repeating the process.

To state our main result on Lagrangian Poincar\'e recurrence, we need
to recall several definitions. Let $U$ be an open subset of a closed,
rational, weakly monotone symplectic manifold $M$. The
\emph{homological capacity} of $U$ is defined as
\begin{equation}
\label{eq:c-hom}
\chom(U)=\sup_F\gamma(\varphi_F),
\end{equation}
where $F$ ranges through all Hamiltonians $F$ supported in $S^1\times U$; see,
e.g., \cite{Sc, Us:ineq, Vi:gen} and references therein. This function
of $U$ has all the expected properties of a symplectic capacity and it
is a standard fact that
\begin{equation}
\label{eq:chom-gamma}
\chom(U)\leq \gamma(\varphi) \textrm{ when } \varphi(U)\cap U
=\emptyset;
\end{equation}
see, e.g., \cite[Prop.\ 3.1]{Us:ineq}. (In applications, sometimes it
is convenient to replace $\gamma(\varphi_F)$ in \eqref{eq:c-hom} by
$\s_{[M]}(F)$; the resulting capacity has the same properties as the
one defined above; see \cite{Gi:We}.) We extend $\chom$ to closed
subsets $L$ of $M$ (for instance, to Lagrangian submanifolds) by
setting
$$
\chom(L)=\inf_U\chom(U),
$$
where the infimum is taken over all open sets $U\supset L$. Note that
$\chom(L)\leq \lambda$; see Remark \ref{rmk:gamma} and the references
therein.

\begin{Example}
\label{lem:chom}
Assume that $L\subset M$ is a closed Lagrangian admitting a Riemannian
metric without contractible closed geodesics. Then
\begin{equation}
\label{eq:chom}
\chom(L)>0.
\end{equation}
The proof of this well-known fact, which we omit here, is quite
standard and implicitly contained in, e.g., the proof of \cite[Thm.\
8.2]{Us:BD}. Moreover, the same is true for certain classes of
coisotropic manifolds (contact type, stable or with totally geodesic
characteristic foliation); see \cite{Gi:coiso, Us:BD}.  Conjecturally,
\eqref{eq:chom} holds for all closed Lagrangians, but this, to the
best of our understanding, is unknown.
\end{Example}

Now we are in a position to state and prove the key result of this
section.

\begin{Theorem}
\label{thm:LPR}
Let $\varphi$ be a pseudo-rotation of $\CP^n$ and let $L\subset \CP^n$
be a closed subset (e.g., a Lagrangian submanifold). Then there exists
a constant $C>0$ and a non-negative integer $d\leq n$, both depending
only on $\vDelta(\varphi)$ but not $L$, such that
$$
\liminf_{k\to\infty} \frac{|\{\ell\leq k \mid \varphi^\ell(L)\cap L\neq
  \emptyset\}|}{k}\geq C\cdot\chom(L)^d.
$$
In particular, the limit inferior is positive when $\chom(L)>0$.
\end{Theorem}

As a consequence, we obtain Theorem \ref{thm:LPR0} from the
introduction:

\begin{Corollary}
\label{cor:LPR}
In the setting of Theorem \ref{thm:LPR}, assume that $\chom(L)>0$
(e.g., $L$ is as in Example \ref{lem:chom}). Then
$\varphi^{k_i}(L)\cap L\neq \emptyset$ for some sequence
$k_i\to\infty$.
\end{Corollary}

\begin{Remark}[Multiplicity of Intersections] 
\label{rmk:mult-inter}
In the general framework of the Lagrangian Poincar\'e recurrence
conjecture, we see no reason to expect a lower bound on the number of
intersections of $\varphi^k(L)$ and $L$ -- after all this is a
dynamics rather than a symplectic topological question. However, in
the setting considered here where the recurrence is a consequence of
the $\gamma$-convergence, the situation is different. Namely, assume
for the sake of simplicity that $\varphi^k(L)$ and $L$ are transverse
for all $k$, which is a generic condition on $L$. Then, in Corollary
\ref{cor:LPR}, the number of intersections is bounded from below by
$\dim\H(L)$, provided that one can replace the upper bound on the
Hofer norm in Chekanov's theorem, \cite{Ch:Lagr}, by an upper bound on
the $\gamma$-norm. Some results in this direction have been recently
announced in \cite{KS}.
\end{Remark}

\begin{proof}[Proof of Theorem \ref{thm:LPR}]
  Without loss of generality we can assume that $\chom(L)>0$ --
  otherwise the assertion is void. Set $\eps=\chom(L)$; then
  $\eps\leq \lambda$. Consider the set
$$
K=\{k\mid \gamma(\varphi^k)<\eps\} \subset \N.
$$
Let $U$ be a neighborhood of $L$. Clearly,
$\chom(U)\geq \eps$ and, by \eqref{eq:chom-gamma}, 
$$
\varphi^k(U)\cap U\neq \emptyset
$$
for every $k\in K$. Since this holds for all $U\supset L$, we
have
 $$
\varphi^k(L)\cap L\neq \emptyset
$$
for all $k\in K$ and the result now follows from Theorem
\ref{thm:gamma}.
\end{proof}

\begin{Remark}[Return Frequency for Small Balls] 
  Let $U$ be a small ball of radius $\delta>0$ in $\CP^n$. Then
  $\chom(U)\geq\const\cdot \delta^2$. Arguing as in the proof of
  Theorem \ref{thm:LPR}, it is easy to see that the return frequency
  for $\varphi$ and $U$ is bounded from below by $C\delta^{2d}$:
\begin{equation}
\label{eq:freq}
\lim_{k\to\infty} \frac{|\{\ell\leq k \mid 
\varphi^\ell(U)\cap U\neq \emptyset\}|}{k}\geq
C\delta^{2d}.
\end{equation}
When $d=n$, \eqref{eq:freq} is exactly the lower bound guaranteed by
the standard Poincar\'e recurrence theorem. In general, $d=n+1-r$,
where $r$ is the number of linearly independent resonance relations
which the mean indices $\hmu(x_i)$ satisfy. Thus \eqref{eq:freq}
provides a stronger lower bound when $r\geq 2$.
\end{Remark}

\begin{Remark}
\label{rmk:packing}
The Lagrangian Poincar\'e recurrence can also be thought of as a
consequence of a hypothetical obstruction to Lagrangian packing in the
same way as the standard Poincar\'e recurrence can be viewed as coming
from the volume obstruction to ball packing. For instance, one might
conjecture that for a compact symplectic manifold $M$ (possibly with
boundary) and a closed Lagrangian $L\subset M$, one can embed into $M$
only a finite number of disjoint Lagrangian submanifolds Hamiltonian
isotopic to $L$. We do not have any counterexamples to this more
general conjecture; nor are we aware of any results in this direction.
\end{Remark}

\begin{Remark}
  It is interesting to compare the Lagrangian Poincar\'e recurrence
  conjecture with Arnold's Legendrian chord conjecture. While at first
  glance the two questions appear to be similar, there are some
  fundamental differences. For instance, in many cases the first
  return time in the chord conjecture is independent of the Legendrian
  submanifold and completely determined by the Reeb flow (see, e.g.,
  \cite{Mo}), but the first return time in the Lagrangian Poincar\'e
  recurrence must, clearly, depend on $L$. (Ultimately, the reason is
  that small, localized, Legendrian submanifolds have localized chords
  unrelated to global dynamics. In other words, the chord conjecture
  is a symplectic topological fact while Lagrangian Poincar\'e
  recurrence relates to dynamics.) On a more technical level,
  Legendrian chords can be treated in the framework of a Floer- or
  SFT-type homology theory (see, e.g., \cite{Ch, EGH}), but no such 
  theory for the Lagrangian Poincar\'e recurrence is known.
\end{Remark}

\subsection{$C^0$-rigidity}
\label{sec:PO}
It turns out that under suitable additional conditions on the mean
index vector $\vDelta$, $\gamma$-convergence in Corollary
\ref{cor:gamma} can be replaced by $C^0$-convergence. For
pseudo-rotations in dimension two this is the $C^0$-rigidity
established in \cite{Br}, and our goal is to extend this result to
higher dimensions. Let us begin with several preliminary observations.

Consider a compact abelian group $\Gamma$ equipped with some
bi-invariant metric. As above we denote the norm of
$\vtheta\in\Gamma$, i.e., the distance to the origin, by $\|\vtheta\|$.

\begin{Definition}
\label{def:exp-L}
A ``vector'' $\vtheta\in\Gamma$ is \emph{exponentially Liouville} if
for every constant $c>0$ there exists $k\in\N$ such that
$\|k\vtheta\|<e^{-c k}$ or, equivalently, a sequence $k_i\to\infty$
satisfying the condition
\begin{equation}
\label{eq:exp-L}
\|k_i\vtheta\|<e^{-c k_i}.
\end{equation}
\end{Definition}

It is clear that when $\Gamma$ is a subgroup of $\Gamma'$, a vector
$\vtheta\in\Gamma$ is exponentially Liouville in $\Gamma$ if and only
if it is exponentially Liouville in $\Gamma'$.  We will apply
Definition \ref{def:exp-L} to the mean index vector $\vDelta$; see
\eqref{eq:vDelta}. But first we need to show that, even though
exponentially Liouville vectors form a zero measure set, they are
generic in the topological sense. The following observation is quite
standard:

\begin{Proposition}
\label{prop:exp-L}
Exponentially Liouville elements of a compact abelian group $\Gamma$
form a residual set $L$ in $\Gamma$,
i.e., $L$ is a countable intersection of open and dense sets.
\end{Proposition}

\begin{proof}
  Note that it is enough to prove the proposition for the connected
  component of the identity in $\Gamma$. Thus we can assume that
  $\Gamma$ is a torus $\T^m=\R^m/\Z^m$. (When $\Gamma$ is finite, the
  assertion is obvious.)

  Following the proof of \cite[Lemma 10]{Br}, consider the collection
  $G_k\subset \T^m$ of points of the form $\vp/k$, where $\vp$ is an
  integer vector and $k\in\N$. This collection is $1/k$-dense in
  $\T^m$ in the obvious sense. Set
$$
U_{k,c}=\bigcup_{g\in G_k}\big\{\vtheta\in \Gamma\,\big|\,
\|k\vtheta-g\|<e^{-c k}\big\},
$$
where $c$ and $k$ are positive integers. This is an open set, which is
also at least $1/k$-dense. (For instance, $G_k\subset U_{k,c}$.) Thus
the union
$$
U_c=\bigcup_{k\in\N} U_{k,c}
$$ is open and dense, and the intersection
$$
L=\bigcap_{c\in \N} U_c
$$
comprising the set of exponentially Liouville elements of $\Gamma$ is
residual.
\end{proof}

One reason we need to consider here a more general group than the
torus $\T^{n+1}$ is that the components of the mean index vector
$\vDelta$ of a pseudo-rotation $\varphi$ must satisfy mean index
resonance relations (see Theorem \ref{thm:index-RR} and \cite{GK}) and
thus $\vDelta$ cannot be a generic vector in $\T^{n+1}$, i.e.,
$\vDelta$ topologically generates a proper subgroup. (If
\eqref{eq:sum-RR} is indeed a universal resonance relation this
subgroup is contained in the anti-diagonal subtorus $\T^n$ in
$\T^{n+1}$.) Furthermore, one might be interested in further
restricting the class of pseudo-rotations and hence of vectors
$\vDelta$.

\begin{Theorem}
\label{thm:exp-L}
Let $\varphi=\varphi_H$ be a pseudo-rotation of $\CP^n$ with
exponentially Liouville mean index vector $\vDelta$. Then there exists
a sequence $k_i\to\infty$ such that
$$
\varphi^{k_i}\stackrel {C^0}{\longrightarrow}
\id .
$$
\end{Theorem}

This is Theorem \ref{thm:exp-L0} from the introduction.

\begin{Remark}
  The sequence $k_i$ is, of course, exactly the sequence of iterations
  such that $\|k_i\vDelta\|<e^{-c k_i}$, where $c$ is completely
  determined by $\| H\|_{C^2}$. For instance, we can take any
  $c>8\| H\|_{C^2}$. Furthermore, the convergence is exponential:
  $\|\varphi^{k_i}\|_{C^0}\leq e^{-a k_i}$ for some $a>0$ which can be
  made arbitrarily large by choosing a large $c$. Finally, note that
  in this theorem we impose no non-degeneracy requirements on~$H$.
\end{Remark}

\begin{proof} The argument closely follows the proof from \cite{Br},
  although we make several shortcuts and use Floer theory instead of
  pseudo-holomorphic curves.
  
  Let us first assume that $\varphi$ is strongly non-degenerate, i.e.,
  all its iterates are non-degenerate. It is easy to see that since
  $\varphi$ is a pseudo-rotation the Floer differential vanishes and
  hence we can identify the Floer homology of $\varphi$ with the Floer
  complex. Let $\bx$ and $\by$ be the capped $k$-periodic orbits
  representing the fundamental class $[\CP^n]$ and the class of the
  point $[\pt]$ in the Floer complex/homology of $\varphi^k$. Denote
  by $\CM(\bx,\by)$ the moduli space of Floer trajectories
  $u\colon S^1_k\times \R \to \CP^n$, where $S^1_k=\R/k\Z$, from $\bx$
  to $\by$. The image $U$ of the evaluation map
$$
\CM(\bx,\by)\to\CP^n, \quad u\mapsto u(0,0)
$$ 
contains an open and dense subset in $\CP^n$. (This is true for any
closed rational symplectic manifold $M$ and any non-degenerate
Hamiltonian when $\bx$ and $\by$ are replaced by action carriers for
$[M]$ and $[\pt]$.) This is an immediate consequence of the standard
fact that for a generic $p\in\CP^n$ the number of $u\in\CM(\bx,\by)$
with $u(0,0)=p$, taken with appropriate signs and viewed as an element
of $\F$, represents the action of $[p]=[\pt]\in\HQ_*(\CP^n)$ on
$\HF_*(\varphi^k)$ and $[\pt]*[\CP^n]=[\pt]$; see
Section~\ref{sec:cap-product}.

Next, recall that
\begin{equation}
\label{eq:velocity-energy}
\big\|\p_s u \big\|_{L^\infty}\leq O\big(E(u)^{1/4}\big),
\end{equation}
where $s$ is the coordinate on $\R$, and the energy $E(u)$ of $u$ is
sufficiently small; see \cite[Sect.\ 1.5]{Sa} or \cite{Br} for a
simple self-contained proof. The upper bound on the right in
\eqref{eq:velocity-energy} is uniform in $k$, and in fact independent
of $k$, and completely determined by the $C^2$-norm of $H$ and the
almost complex structure $J$. (It is essential here that we view the
iterated flow not as $\varphi_H^{kt}$ but as the flow $\varphi_H^t$
with $t\in [0,k]$.)

Denote by $d$ the distance in $\CP^n$. We claim that
\begin{equation}
\label{eq:distance}
d\big(p,\varphi^k(p)\big) \leq e^{Ck} O\big(E(u)^{1/4}\big)
\end{equation}
for every $p\in \CP^n$, where we can take any $C>2\| H\|_{C^2}$,
provided that $E(u)$ is small enough. To prove this, assume first that
$p\in U$. Set $z(t)=u(0,t)$ for some $u$ with $p=u(0,0)$ and
$\zeta(t)=\big(\varphi^t\big)^{-1}\big(z(t)\big)$. Then
$$
\dot{z}(t)=X_H\big(z(t)\big)+ D\varphi^t\big(\dot{\zeta}(t)\big),
$$
and hence
$$
\dot{\zeta}(t)=\big(D\varphi^t\big)^{-1}\big(
\dot{z}(t)-X_H(z(t))\big).
$$
It is easy to see that 
\begin{equation}
\label{eq:exp}
\big\|\big(D\varphi^t\big)^{-1}\big\|\leq e^{C_0t},
\end{equation}
where we can take $C_0=\|H\|_{C^2}$. Moreover, from the Floer equation
and \eqref{eq:velocity-energy}, we infer that
$$
\big\|X_H\big(z(t)\big)-\dot{z}(t)\big\|\leq O\big(E(u)^{1/4}\big),
$$
for all $t\in [0,k]$. Therefore,
$$
\big\| \dot{\zeta}(t)\big\|\leq e^{C_0t} O\big(E(u)^{1/4}\big)
$$
 and thus
$$
d\big(\zeta(0),\zeta(k)\big)\leq e^{C_0 k} O\big(E(u)^{1/4}\big).
$$
Clearly, $\zeta(0)=p$ and, since $z$ is a loop,
$\zeta(k)=\varphi^{-k}(p)$. Applying $\varphi^k$ to $\zeta(0)$ and
$\zeta(k)$ and using \eqref{eq:exp} again, we obtain
\eqref{eq:distance} with $C>2C_0$. Finally, the upper bound is uniform
in $p$ and $U$ is dense in $\CP^n$. Therefore, \eqref{eq:distance}
holds on $\CP^n$.

Next, let us find an upper bound on $E(u)$. This is where the
requirement that $\vDelta$ is exponentially Liouville enters the proof
and the argument is quite similar to the end of the proof of Theorem
\ref{thm:gamma}.  As in that proof, it is convenient to rescale the
symplectic structure on $\CP^n$ so that $[\omega]=2c_1(T\CP^n)$ and
normalize $H$ to ensure that $\CS(H)=\CSI(\varphi)$. Then, in
particular,
$$
\CA_{H^{\nat k}}(\bx)=\hmu(\bx) \textrm{ and } 
\CA_{H^{\nat k}}(\by)=\hmu(\by).
$$

Assume that $k=k_i$ for the sequence $k_i$ from Definition
\ref{def:exp-L} with $\vtheta=\vDelta$ and a sufficiently large $c$ to
be specified later. Then, by \eqref{eq:exp-L}, the spectrum
$\CS\big(H^{\nat k}\big)=\CSI(\varphi^k)$ is located in a neighborhood
of $2(n+1)\Z$ of size $e^{-ck}$. Arguing exactly as in the proof of
Theorem \ref{thm:gamma}, it is easy to show that $\hmu(\bx)$ and
$\hmu(\by)$ both lie in the cluster centered at 0. Thus

$$
|\hmu(x)|\leq e^{-ck}\textrm{ and } |\hmu(y)|\leq e^{-ck}
$$
and
$$
E(u)=\CA_{H^{\nat k}}(\bx)-\CA_{H^{\nat
    k}}(\by)=\hmu(\bx)-\hmu(\by)\leq 2e^{-ck}.
$$

Set $c>4C$. (For instance, $C=2.1\|H\|_{C^2}$ and
$c=8.5\| H\|_{C^2}$.)  Then, by \eqref{eq:distance} for $k=k_i$,
we have
$$
d\big(p,\varphi^{k_i}(p)\big) \leq O\big(e^{(C-c/4)k_i}\big)\to 0
$$
as $k_i\to\infty$. Thus
$$
\|\varphi^{k_i}\|_{C^0}\to 0,
$$
which proves the theorem when $\varphi$ is strongly non-degenerate.

Dealing with the degenerate case, consider a $C^2$-small
non-degenerate perturbation $F$ of $H^{\nat k}$. As above, let $\bx$
and $\by$ be the action carriers for $[\CP^n]$ and, respectively,
$[\pt]$.

\begin{Lemma}
\label{lemma:deg-energy}
Assume that $F$ is sufficiently $C^2$-close to $H^{\nat k}$. Then, for
a generic point $p\in\CP^n$, there exists a solution $u$ of the Floer
equation for $F$ with $u(0,0)=p$ such that
\begin{equation}
\label{eq:deg-energy}
E(u)\leq 2 \big( \CA_{H^{\nat k}}(\bx)- \CA_{H^{\nat k}}(\by)\big).
\end{equation}
\end{Lemma}

\begin{proof} We start with several general remarks. Fix
  $\eps>0$. Then, when $F$ is sufficiently $C^2$-close to
  $H^{\nat k}$, every capped $k$-periodic orbit $\bz$ of $H^{\nat k}$
  splits under the perturbation $F$ into several capped orbits $\bz_i$
  located near $\bz$ with actions and mean indices $\eps$-close to the
  action and the mean index of $\bz$. All $k$-periodic orbits of $F$
  arise in this way.  (We view both $F$ and $H^{\nat k}$ as
  $k$-periodic Hamiltonians.) Moreover, for two $k$-periodic orbits
  $\bz$ and $\bz'$ we have
$$
\CA_{H^{\nat k}}(\bz)>\CA_{H^{\nat k}}(\bz')
\textrm{ iff }
\CA_{F}(\bz_i)>\CA_{F}(\bz'_j)
\textrm{ for all (or just one pair of) $i$ and $j$}.
$$

Let $\bz$ be the action carrier for some class
$\alpha\in\HF\big(H^{\nat k}\big)$. Then, in the notation from Section
\ref{sec:spec} and, in particular, \eqref{eq:cycle-action}, for every
cycle $\sigma\in \CF_*(F)$ representing~$\alpha$,
\begin{equation}
\label{eq:A}
\s_\sigma(F)\geq \s_\alpha\big( H^{\nat k}\big)-\eps=\CA_{H^{\nat
    k}}(\bz)-\eps,
\end{equation}
and there exists a representative $\sigma_{\min}$ such that
\begin{equation}
\label{eq:Amin}
\s_{\sigma_{\min}}(F)\leq \s_\alpha\big( H^{\nat k}\big)+\eps=\CA_{H^{\nat
    k}}(\bz)+\eps.
\end{equation}
In particular, at least one of the orbits $\bz_i$ enters
$\sigma_{\min}$ with non-zero coefficient.

Let us pick such a cycle $\sigma_{\min}$ representing $[\CP^n]$ for
$F$ and satisfying \eqref{eq:Amin}. Thus
$$
\s_{\sigma_{\min}}(F)\leq \CA_{H^{\nat
    k}}(\bx)+\eps.
$$

Next, we pick a generic $p\in\CP^n$ and view it as a cycle
representing $[\pt]\in \HQ_*(\CP^n)$. Then acting by $p$ on
$\sigma_{\min}$, we obtain a cycle $P\in \CF_{-n}(F)$ also
representing $[\pt]$. By \eqref{eq:A},
$$
\s_P(F)\geq \CA_{H^{\nat
    k}}(\by)-\eps.
$$
A point $\by_j$ with action $\s_P(F)$ enters the cycle $P$, and hence
there exists a solution $u$ of the Floer equation for $F$ connecting a
point from $\sigma_{\min}$ to $\by_j$. This solution has energy
$$
E(u)\leq \s_{\sigma_{\min}}(F)-\s_P(F)\leq \CA_{H^{\nat
    k}}(\bx)-\CA_{H^{\nat k}}(\by)+2\eps
$$
and passes through $p$, i.e., $u(0,0)=p$. Making $\eps>0$ small enough
we obtain \eqref{eq:deg-energy}, which completes the proof of the
lemma.
\end{proof}

Now the proof of the theorem is finished essentially in the same way
as in the non-degenerate case. First, note that $\| F\|_{C^2}$ can be
made arbitrarily close to $\| H^{\nat k}\|_{C^2}=\| H\|_{C^2}$. (It is
again essential here that $H$ is periodic in time and $H^{\nat k}_t$
is simply the Hamiltonian $H_t$ but with $t\in S^1_k$.)  Then
\eqref{eq:velocity-energy} still holds, where the upper bound on the
right is uniform in $k$ and completely determined by $\| H\|_{C^2}$.

Next, the bound \eqref{eq:distance} turns into
$$
d\big(p,\varphi_F(p)\big) \leq e^{Ck} O\big(E(u)^{1/4}\big)
$$
for a generic $p$, where we can again take any $C>2\| H\|_{C^2}$. By
making $F$ sufficiently close to $H^{\nat k}$ we can make sure that
$d\big(\varphi^k(p),\varphi_F(p)\big)$ is arbitrarily small for all
$p$, and hence \eqref{eq:distance} holds in its original form for all
$p$ uniformly in $k$, provided that $E(u)$ is small enough (depending
on $\| H\|_{C^2}$).

Finally, when $k=k_i$, we have
$$
E(u)\leq 2\big(\CA_{H^{\nat k}}(\bx)-\CA_{H^{\nat
    k}}(\by)\big)=2\big(\hmu(\bx)-\hmu(\by)\big)\leq 4e^{-ck}
$$
by Lemma \ref{lemma:deg-energy} and again
$ \|\varphi^{k_i}\|_{C^0}\to 0 $ exponentially fast when $c>4C$.
\end{proof}

\begin{Remark}
  As has been pointed out to us by Seyfaddini, one can also derive
  Theorem \ref{thm:exp-L} from Theorem \ref{thm:gamma} and Remark
  \ref{rmk:Delta-gamma} by employing a ``H\"older type'' inequality
  relating the $C^0$-norm, the $\gamma$-norm and the $C^0$-norm of the
  derivative. However, the direct proof given here is of independent
  interest and may have other applications.
\end{Remark}

Just as in \cite{Br} we have the following corollary of Theorem
\ref{thm:exp-L}:

\begin{Corollary}
\label{cor:mixing}
Let $\varphi$ be a pseudo-rotation of $\CP^n$ with exponentially
Liouville mean index vector. Then $\varphi$ is not topologically
mixing and, in particular, not mixing with respect to the Lebesgue
measure.
\end{Corollary}

We conclude this section with several remarks.  First, note that in
dimension two the $\gamma$-norm is continuous with respect to the
$C^0$-norm; \cite{Se}.  Furthermore, similar results in higher
dimensions for symplectically aspherical manifolds and in some other
settings have been recently obtained in \cite{BHS}. Thus, for $S^2$
and more generally when such continuity is established, Theorem
\ref{thm:exp-L} implies Corollary \ref{cor:gamma} in the exponentially
Liouville case. Finally, we conjecture that a variant of Theorem
\ref{thm:exp-L} holds without the assumption that $\Delta$ is
exponentially Liouville and in this case one may also have an analog
of the frequency bound similar to that in Theorem~\ref{thm:gamma}.

It is also worth mentioning that strictly speaking the results in
\cite{Br} are established for exponentially Liouville pseudo-rotations
of $D^2$ while, for $n=1$, our Theorem \ref{thm:exp-L} concerns
exponentially Liouville pseudo-rotations of $S^2$. The difference in
the domains is rather technical than conceptual, although it does
effect what symplectic topological tools are better suited for the
task (holomorphic curves vs.\ Floer homology). In any event, the
results for $S^2$ can be derived from those for $D^2$ and vice
versa. In one direction, from $D^2$ to $S^2$, this can be done by
simply applying an oriented real blow-up to one of the fixed
points. In the opposite direction, the argument is considerably more
subtle and requires more work. The difficulty lies in the fact that a
pseudo-rotation of $S^2$ obtained from one of $D^2$ by, e.g.,
collapsing $\p D^2$ to a point or doubling $D^2$, is not
smooth. However, it is Lipschitz and the proof seems to go through
with only minor modifications including extending the resonance
relations to this setting. Finally, the requirement in \cite{Br} that
the pseudo-rotation is irrational is more of terminological than of
mathematical nature: a pseudo-rotation of $D^2$ is automatically
irrational, for otherwise it would have periodic orbits on $\p D^2$.
One can interpret this requirement as an implicit non-degeneracy
condition. Then similarly, it is also automatically satisfied for
pseudo-rotations of $S^2$; see Corollary \ref{cor:S^2-nondeg}.

\section{Crossing energy and the proof of Theorem \ref{thm:isolated}}
\label{sec:energy-isolated}

The proof of Theorem \ref{thm:isolated} hinges on a technical result
generalizing \cite[Thm.\ 3.1]{GG:hyperbolic} and asserting, roughly
speaking, that the energy of a Floer trajectory asymptotic to $x^k$
and crossing a fixed isolating neighborhood of $x$ is bounded from
below by a constant independent of $k$. With future applications in
mind we start by proving this result in a form more general than
needed for the proof of the theorem.

\subsection{Crossing energy}
\label{sec:cross-energy}
Let $K\subset M$ be a compact invariant set of a Hamiltonian
diffeomorphism $\varphi=\varphi_H$ of a symplectic manifold
$M$. Recall that $K$ is said to be isolated (as an invariant set) if
there exists a neighborhood $U\supset K$ such that for no initial
condition $p\in U\setminus K$ the orbit through $p$ is contained in
$U$, i.e., there exists $k\in\Z$, possibly depending on $p$, such that
$\varphi^k(p)\not\in U$. The neighborhood $U$ is called an isolating
neighborhood of $K$. Then any neighborhood of $K$ contained in $U$ is
also isolating, and hence such neighborhoods can be made arbitrarily
small. (Note that a fixed point may be isolated for all iterations as
a fixed point while not isolated as an invariant set.)

Consider solutions $u\colon \Sigma\to M$ of the Floer equation for
$H^{\nat k}$, where $\Sigma\subset \R\times S^1_k$ is a closed domain,
i.e., a closed subset with non-empty interior. Note that the period
$k$ is not fixed, and the domain $\Sigma$ of $u$ need not be the
entire cylinder $\R\times S^1_k$.  By definition, the energy of $u$ is
$$
E(u)=\int_\Sigma \|\p_s u\|^2 \, ds dt.
$$
Here $\|\cdot\|$ stands for the norm with respect to
$\left<\cdot\,,\cdot\right>=\omega(\cdot,J\cdot)$, and hence
$\|\cdot\|$ also depends on the background almost complex structure
$J$. We say that $u$ is asymptotic to $K$ as $s\to\infty$ (or as
$s\to-\infty$) if for any neighborhood $V$ of $K$ the domain $\Sigma$
contains a cylinder $[s_V,\,\infty)\times S^1_k$ (or
$(-\infty,\,s_V]\times S^1_k$) which is mapped into $V$ by $u$.

Finally, let $U$ be a fixed (sufficiently small) isolating
neighborhood of $K$. Set $\p U:=\bar{U}\setminus U$.

\begin{Theorem}[Crossing Energy Theorem]
\label{thm:energy}
There exists a constant $c_\infty>0$, independent of $k$ and $\Sigma$,
such that for any solution $u$ of the Floer equation with
$u(\p \Sigma)\subset \p U$ and $\p \Sigma\neq\emptyset$, which is
asymptotic to $K$ as $s\to\infty$ or $s\to-\infty$, we have
\begin{equation}
\label{eq:energy}
E(u)>c_\infty .
\end{equation}
Moreover, the constant $c_\infty$ can be chosen to make
\eqref{eq:energy} hold for all $k$-periodic almost complex structures
(with varying $k$) $C^\infty$-close to $J$ uniformly on $\R\times U$.
\end{Theorem}

The key point of this result is that the lower bound $c_\infty>0$ can
be taken independent of $k$. The requirements of the theorem are met
when $K$ is just one fixed point $x$ of $\varphi$ and $x$ is
hyperbolic -- this is the setting of \cite[Thm.\ 3.1]{GG:hyperbolic}
-- or more generally when $K=\{x\}$ is isolated as an invariant set as
in Theorem \ref{thm:isolated}. The requirements are also satisfied
when $K$ is a hyperbolic invariant subset (e.g., a Smale's horseshoe).
The condition that $K$ is isolated is essential and cannot be removed;
see \cite[Rmk.\ 3.4]{GG:hyperbolic}.  Also note that periodic orbits
of $\varphi_H^t$ in $M$ corresponding to the fixed points in $K$ need
not be contractible. The proof of Theorem \ref{thm:energy} follows
closely the line of reasoning establishing \cite[Thm.\
3.1]{GG:hyperbolic}. However, Theorem \ref{thm:energy} is considerably
more general than that result and we feel that giving a detailed proof
is justified.

\begin{proof}
  We will focus on solutions $u$ asymptotic to $K$ as
  $s\to\infty$. The case of $s\to-\infty$ can be handled in a similar
  fashion. Before going into details of the proof, let us spell out
  the idea. Note first that since $U$ is an isolating neighborhood,
  there exists $T_0>1$, an escape time, such that every integral curve
  $\varphi_H^t(p)$, $t\in [-T_0,\, T_0]$, touching $\p U$ at some
  moment $\tau\in [0,\,1]$ cannot be entirely contained in
  $\bar{U}$. Arguing by contradiction, assume that there exists a
  sequence of solutions $u_i$ of the Floer equation with $E(u_i)\to 0$
  asymptotic to $K$ at $+\infty$ and defined on
  $\Sigma_i\subset \R\times S^1_{k_i}$. Then for every $T>T_0$ one can
  also find a sequence of solutions $v_i$ defined on a subset of a
  rectangle $[-a,\,a]\times [-T,\,T]\subset \C$ containing
  $[-a,\,a]\times [0,\,T]$, mapping this half-rectangle into $\bar{U}$
  and such that $v_i(0, \tau_i)\in \p U$ for some $\tau_i\in [0,\,1]$.
  As $i\to\infty$, these solutions converge to an integral curve
  $(s,t)\mapsto \varphi_H^t(p)$ in the sense of the target-local
  compactness theorem from \cite{Fi}. This curve is parametrized by
  $[-T',\,T']$ with $T>T'>T_0$, touches $\p U$ at some moment
  $\tau\in [0,1]$ and is entirely contained in $\bar{U}$, which is
  impossible by the definition of $T_0$.

  Throughout the proof, it will be convenient to work in $M\times S^1$
  rather than in $M$. Thus let $\tvarphi^t$ be the flow on
  $M\times S^1$ induced by the isotopy $\varphi_H^t$, i.e.,
  $\tvarphi^t(p,\theta)=(\varphi_H^t(p),\theta+t\, \mathrm{mod}\,
  1)$. This is indeed a true flow since $H$ is one-periodic in
  time. In a similar vein, the map $u$ gives rise to the map
$$
\tu\colon \Sigma_i\to M\times S^1, \quad (s,t)\mapsto
\big(u(s,t),t\,\mathrm{mod}\, 1\big).
$$ 
Set $\tK=K\times S^1$, and let $\tB$ and $\tU$ be closed isolating
neighborhoods of $\tK$ such that $\tB\subset \mathrm{int}(\tU)$. (For
instance, we can initially take $\tU=\bar{U}\times S^1$.)

Arguing by contradiction, assume that there exists a sequence of
iterations $k_i\to\infty$, a sequence of $k_i$-periodic almost complex
structures $J_i$ on $M$, compatible with $\omega$ and
$C^\infty$-converging to $J$ uniformly on $\R\times U$, and a sequence
$u_i\colon\Sigma_i\to U$ of solutions of the Floer equation for $J_i$
and $H^{\nat k_i}$, satisfying the hypotheses of Theorem
\ref{thm:energy} and such that $E(u_i)\to 0$.

To proceed, let us first make several simplifying assumptions. Namely,
without loss of generality we can assume that the boundaries $\p \tU$
and $\p \Sigma_i$ are smooth. Indeed, to this end we can shrink $\tU$
slightly and simultaneously make sure that $\p \tU$ is transverse to
the maps $\tu_i$. At this stage we do not need $\tU$ to be a direct
product.

Furthermore, we can assume that $[0,\infty)\times S^1_{k_i}$ is the
largest half-cylinder in $\Sigma_i$ mapped by $\tu_i$ into $\tB$,
i.e., $\tu_i\big([0,\infty)\times S^1_{k_i}\big)\subset \tB$ and
$\tu_i\big(\{0\}\times S^1_{k_i}\big)$ touches $\p \tB$ at at least
one point $\tu_i(0,\tau_i)$ with $0\leq \tau_i\leq 1$. Here the first
assertion readily follows since $H$ is independent of $s$, and hence
the Floer equation is translation invariant. As a consequence, we can
change $u_i$ by applying a translation in $s$ without affecting the
energy. To ensure that $0\leq \tau_i\leq 1$, we apply an integer
translation in $t$ to $u_i$ and $J_i$. Since the almost complex
structures $J_i$ $C^\infty$-converge to $J$ uniformly in $t\in\R$, the
same is true for the translated almost complex structures. (This
changes the almost complex structures $J_i$ and the solutions $u_i$ by
translation, but again does not affect the energy of $u_i$.)

Finally, by passing if necessary to a subsequence, we may assume that
the sequence $\tau_i$ converges.

As has been pointed out above, since $\tB$ is an isolating
neighborhood, there exists a constant $T_0>1$ (an escape time),
depending only on $\tB$ and $H$, such that no integral curve of
$\tvarphi^t$ passing through a point of $\p \tB$ (or near $\p \tB$) at
a moment $\tau\in [0,\,1]$ can stay in $\tB$ for all $t$ with
$|t|<T_0$. This observation is the key to the proof and this is where
we use the condition that $K$ is isolated.

Next, factoring the universal covering map $\C\to\R\times S^1$ as
$$
\C\to\R\times S^1_{k_i}\to\R\times S^1,
$$
we lift the domains $\Sigma_i$ to the domains $\hat{\Sigma}_i$ in $\C$
and view the maps $u_i$ as maps $\hat{\Sigma}_i\to M$, which are
$k_i$-periodic in $t$. The graph $\Gamma_i$ of $u_i$ is an embedded
$\hat{J}_i$-holomorphic curve in $M\times \C$ with respect to an
almost complex structure $\hat{J}_i$ which incorporates both $J_i$ and
$X_H$. The projection $\pi\colon M\times\C\to \C$ is holomorphic, and
hence so is the projection of $\Gamma_i$ to $\C$. Let
$\tau=\lim\tau_i\in [0,\,1]$.

Pick arbitrary constants $T>T_0$ and $a>0$ and set
$$
\Pi=[-a,\,a]\times [- T,\, T]\subset \C.
$$
From now on we will focus on the restrictions $v_i:=u_i|_\Pi$ and
$\tv_i:=\tu_i|_\Pi$. Let $S_i$ be the graph of $v_i$.  Denoting by
$\tU_T$ the part of the lift of $\tU$ to the covering
$M\times \R\to M\times S^1$ lying over $[-T,\,T]\subset \R$, we have
$$
S_i=\Gamma_i\cap P,\textrm{ where } P=\tU_T\times [-a,\,a]\subset
M\times \C.
$$
Clearly, $\p S_i\subset \p P$ and
$$
\Area(S_i)\leq\Area(\Pi)+E(u_i)<\const,
$$ 
where the constant on the right is independent of $i$. Let us now
shrink $\Pi$ and $\tU$ slightly. To be more precise, set
$$
\Pi'=[-a',\,a']\times [-T',\,T']\subset \Pi,
$$
where $0<a'<a$ and $T_0<T'<T$, and let $\tU'$ be a closed neighborhood
of $K\times S^1$ with boundary close to $\p \tU$. (More specifically,
$\tB\subset \mathrm{int}(\tU')$ and $\tU'\subset \mathrm{int}(\tU)$.)
We denote the resulting subset of $M\times \C$ by $P'$.

By the target-local Gromov compactness theorem, \cite[Thm.\ A]{Fi},
the intersections of the embedded $\hat{J}_i$-holomorphic curves
$S_i\cap P'$ Gromov--converge, after passing if necessary to a
subsequence, to a (cusp) $\hat{J}$-holomorphic curve $S'$ in $P'$ with
boundary in $\p P'$. This holomorphic curve is a union of
multi-sections over subsets of $\Pi'$ and possibly some components
contained in the fibers of the projection $\pi\colon P'\to \Pi'$ (the
bubbles). The latter are points since $E(u_i)\to 0$.

Furthermore, $S'$ is in fact a unique section over some subset $D$ of
$\Pi'$. One way to see this is to observe that the intersection index
of $S'$ with the fiber over a regular point $(s,t)$ of its projection
to $\Pi'$ is either one or zero -- the intersection index of $S_i$
with the fiber. For instance, the index is one when $(s,t)$ is in the
domain of each $v_i$ and the distance from $\tv_i(s,t)$ to $\p \tU'$
stays bounded away from zero as $i\to \infty$. Note also that here we
need the parameters $a'$ and $T'$ and the neighborhood $\tU'$ to be
``generic''.

To summarize, $S'$ is the graph of a solution $v$ of the Floer
equation defined on some connected subset $D$ of $\Pi'$. Moreover, as
is easy to see from \cite{Fi}, after making an arbitrarily small
change to $a'$ and $T'$ we can ensure that the domain $D$ of $v$ has
piece-wise smooth boundary. The maps $u_i$ uniformly converge to $v$
on compact subsets of $\mathrm{int}(D)$.

The domain $D$ contains the half-rectangle $\Pi'_+= \{s>0\}\cap\Pi'$.
Indeed, $\tu_i$ maps $\Pi_+$, the $\{s\geq 0\}$-part of $\Pi$, into
$\tB\subset \textrm{int}(\tU')$. Hence the projection of $S'_j$ to
$\Pi$ contains $\Pi_+$ or, in other words, $\Pi_+$ is in the domain of
$v$. Furthermore, $D$ also contains the closure of $\Pi'_+$ and, in
particular, the point $(0,\tau)$. In fact,
$(0,\tau)\in\textrm{int}(D)$ since the distance from the points
$\tv_i(0,\tau)\in \tB$ to $\p \tU'$ stays bounded away from zero. Let
$\tv$ be the natural lift of $v$ to a map to $\tU'$. Then
$$
\tv(0,\tau)=p:=\lim \tv_i(0,\tau_i)\in \p \tB.
$$

Since $E(u_i)\to 0$, we have $E(v)=0$. Thus $\p_s v(s,t)=0$
identically on $D$, and hence $v(s,t)$ is an integral curve
$\gamma(t)$ of the flow $\tvarphi^t$ on $M\times S^1$. This integral
curve passes through the point $p\in \p \tB$ at the moment $\tau$, and
$\gamma(t)\in \tB$ for all $t\in [-T',\,T']$, which is impossible due
to our choice of $T_0$ and the fact that $T'>T_0$. This contradiction
completes the proof of the theorem.
\end{proof}

\subsection{Proof of Theorem \ref{thm:isolated}}
\label{sec:pf-isolated}
With the crossing energy lower bound established, we are now in a
position to prove Theorem \ref{thm:isolated}. In fact, we prove a
slightly more precise result. To state it, recall that the normalized
augmented action of a $k$-periodic orbit $y^k$ is simply
$\tCA_{H^{\nat k}}(y^k)/k$, where the augmented action is defined by
\eqref{eq:aug-action}. Thus all iterations of an orbit have the same
normalized augmented action.

\begin{Theorem}
\label{thm:isolated2}
Let $M^{2n}$ be a strictly monotone symplectic manifold.
Assume that $N\geq n+1$ and 
\begin{equation}
\label{eq:hom-relation}
\alpha *\beta = \qq [M]
\end{equation}
in $\HQ_*(M)$ for some homology classes $\alpha\in \H_*(M)$ and
$\beta\in \H_*(M)$ with $|\alpha|<n$ and $|\beta|<n$.  Let $\varphi_H$
be a Hamiltonian diffeomorphism of $M$ with a contractible periodic
orbit $x$ which has a neighborhood not intersecting any other periodic
orbit of any period, is isolated as an invariant set and such that
$\HF(x^k)\neq 0$ for all $k\in\N$. Furthermore, let $\tI$ be an
arbitrary interval containing $\tCA_H(x)$.  Then $\varphi$ has
infinitely many periodic orbits with normalized augmented action in
$\tI$.
\end{Theorem}

Theorem \ref{thm:isolated} readily follows from this result. Some
remarks are due before the proof of the theorem.

\begin{Remark}
\label{rmk:isolated}
The key new point of Theorem \ref{thm:isolated2} when compared to
Theorem \ref{thm:isolated} is the control of the augmented action. A
similar refinement can also be made in \cite[Thm.\
1.1]{GG:hyperbolic}.  Theorem \ref{thm:isolated2} and, as a
consequence, Theorem \ref{thm:isolated} have analogs when the orbit
$x$ is not contractible. In this case, we need to require $M$ to be
toroidally monotone with $N$ being the ``toroidal'' minimal Chern
number; cf.\ \cite[Rmk.\ 1.2 and 4.4]{GG:hyperbolic}. Then the
periodic orbits detected in the theorem lie in the free homotopy
classes of the iterated orbits $x^k$, $k\in\N$.
\end{Remark}

\begin{Remark}
  As has been pointed out in Section \ref{sec:sets}, the only monotone
  manifold with $N\geq n+1$ known to the authors is $\CP^n$. There are
  numerous negative monotone manifolds, simply connected and not,
  meeting the conditions of the theorem. Although in this case the
  Conley conjecture holds and $\varphi$ has infinitely many periodic
  orbits unconditionally (see \cite{CGG, GG:nm}), Theorem
  \ref{thm:isolated2} gives additional information about the augmented
  actions of the orbits or their location and free homotopy classes;
  cf.\ \cite{Ba2, GG:nc, Gu:nc, Or1, Or2}.
\end{Remark}

\begin{proof}[Proof of Theorem \ref{thm:isolated2}]
  Before giving a detailed argument let us lay out the logic of the
  proof and explain the main idea. To this end, we assume that $H$ is
  strongly non-degenerate and, arguing by contradiction, that it has
  finitely many periodic orbits; cf.\ the proof of \cite[Thm.\
  1.1]{GG:hyperbolic}. Then, when $U$ is small enough, none of these
  orbits other than $x$ enter $U$. Fix an arbitrary capping of $x$. By
  adding a constant to $H$, we can ensure that
  $\CA_H(\bx)=0$. Furthermore, set $I=(-c,\,c)$, where $c>c_\infty$,
  and pick $\eps<c_\infty$.

  One can then show that there exists an arbitrarily large $k$ such
  that $\bx^k$ is not connected to any $k$-periodic orbit with action
  in $I$ by a solution of the Floer equation of relative index one and
  that all capped $k$-periodic orbits have action in the
  $\eps$-neighborhood of $\lambda\Z$. Acting by $\alpha$ and $\beta$
  on $[\bx^k]\in \HF_*^I\big(H^{\nat k}\big)$ we find a capped
  $k$-periodic orbit $\bar{y}$ and two solutions, $u$ and $v$, of the
  Floer equation for $H^{\nat k}$, the first of which is connecting
  $\bx^k$ to $\by$ and the second is from $\by$ to $\bx^k\#
  A_0$. Here, as in Section \ref{sec:FFH}, $A_0$ is the generator of
  $\Gamma=\pi_2(M)/\ker I_\omega$, and thus
  $\left<\omega,A_0\right>=\lambda$ and
  $\left< c_1(TM),A_0\right> = 2N$. The action of $\by$ is either in
  $(-\eps,\,\eps)$ or in $(-\eps,\,\eps)-\lambda$, and hence either
  $E(u)<\eps$ or $E(v)<\eps$, which is impossibly by Theorem
  \ref{thm:energy} due to the assumption that $\eps<c_\infty$.

  Let us turn now to a rigorous proof in complete generality. As
  above, arguing by contradiction, assume that $H$ has finitely many
  periodic orbits with normalized augmented action in $\tI$. Then,
  replacing $H$ by an iterate and making a change of time, we can also
  assume that $x$ and all simple periodic orbits of $H$ with action in
  $\tI$ are one-periodic. We denote these orbits by
  $x=x_0,\, x_1,\ldots,x_r$. As above, let us pick an arbitrary
  capping of $x$, denote by $\bx$ the resulting capped orbit and
  adjust $H$ so that $\CA_H(\bx)=0$.

  To further fix notation, set $a_i=\CA_H(x_i)\in S^1_\lambda$ and
  $\Delta_i=\hmu(x_i)\in S^1_{2N}$. Thus here again we view the
  actions and mean indices of the orbits without capping as elements
  of the circles $S^1_\lambda=\R/\lambda\Z$ and, respectively,
  $S^1_{2N}=\R/2N\Z$. We have $a_0=0$ due to the normalization of
  $H$. Also set
$$
\ta_i=\tCA_{H}(x_i)\in\tI\subset\R.
$$
Clearly, $\CA_{H^{\nat k}} (x_i^k)=ka_i$ and $\hmu(x_i^k)=k\Delta_i$
and $\tCA_{H^{\nat k}} (x_i^k)=k\ta_i$. Finally, note that without
loss of generality we can assume that $\tI=(\ta_0-\eta,\, \ta_0+\eta)$
for some $\eta>0$.

Fix a one-periodic in time almost complex structure $J^0$. Let $U$ be
an isolating neighborhood of $x$ such that, in addition, no periodic
orbit of $\varphi_H$ with normalized augmented action in $\tI$ other
than $x$ enters $U$. By Theorem \ref{thm:energy} applied to
$K=\{x(0)\}$, there exists a constant $c_\infty>0$ such that, for all
$k\in\N$, every non-trivial $k$-periodic solution of the Floer
equation for the pair $(H,J)$ asymptotic to $x^k$ as $s\to\pm\infty$
has energy greater than~$c_\infty$, where $J$ is $k$-periodic and
sufficiently $C^\infty$-close to $J^0$. (In what follows, such an
almost complex structure $J$ is chosen to be generic and is suppressed
in the notaion.)

Let $c>0$ be outside the union of the action spectra
$\CS\big(H^{\nat k}\big)$, $k\in \N$, and set $I=(-c,\,c)$.  Next we
pick a large constant $C>0$ and a small constant $\eps>0$ depending on
$c$, to be specified later. As is easy to show using the Kronecker
theorem, there exists an arbitrarily large $k\in\N$ such that for all
$i$
\begin{equation}
\label{eq:T1}
\| k a_i\|_{\lambda}<\eps
\end{equation}
and 
\begin{equation}
\label{eq:T2}
\textrm{either }\ta_i=\ta_0\textrm{ or }k|\ta_i-\ta_0|>C.
\end{equation}
Here $\|a\|_{b}\in [0,\, b/2]$ stands for the distance from
$a\in S^1_{b}=\R/b\Z$ to $0$. We will also require that
\begin{equation}
\label{eq:eta-C}
k>C/\eta.
\end{equation}

\begin{Lemma}
\label{lemma:homology}
Assume that $C>0$ is sufficiently large and $\eps>0$ is sufficiently
small and that $k\in\N$ satisfies the requirements \eqref{eq:T1},
\eqref{eq:T2} and \eqref{eq:eta-C}, and let $F$ be a $k$-periodic,
non-degenerate $C^2$-small perturbation of $H^{\nat k}$. Then, in the
notation from Section \ref{sec:LFH},
$\CF_*\big(F,\bx^k\#(\ell A_0)\big)$ is a direct summand in
$\CF_*^I(F)$ whenever
$\CA_{H^{\nat k}}\big(\bx^k\#(\ell A_0)\big)\in I$. In particular,
$$
\bigoplus_{|\ell|<c/\lambda} \HF_*\big(\bx^k\# (\ell A_0)\big)
$$
is a direct summand in $\HF_*^I\big(H^{\nat k}\big)$
\end{Lemma}

The specific bounds on $C>0$ and $\eps>0$ are \eqref{eq:C},
\eqref{eq:eps1} and \eqref{eq:eps2} below.

\begin{proof}
  Set $\bx_k=\bx^k\# (\ell A_0)$, where $|\ell|<c/\lambda$ as in the
  assertion of the lemma, and hence $\CA_{H^{\nat k}}(\bx_k)\in
  I$. Let $\by_k$ be any other capped $k$-periodic orbit with action
  in $I$. We claim that $\by_k$ and $\bx_k$ can be connected by a
  solution of the Floer equation only when
\begin{equation}
\label{eq:diff-index}
\big|\hmu(\bx_k)-\hmu(\by_k)\big|> 2n+1.
\end{equation}
Once this is established, Lemma \ref{lemma:homology} will follow from
Lemma \ref{lemma:spec}.

Assume first that $y_k$ is one of the orbits $x_i^k$, e.g.,
$y_k=x_1^k$. Then, by \eqref{eq:T2}, there are two cases to consider.

If $k|\ta_1-\ta_0|>C$, we have
$$
\big|\CA_{H^{\nat k}}(\bx_k)-\CA_{H^{\nat k}}(\by_k)\big|
+\frac{\lambda}{2N}\big|\hmu(\bx_k)-\hmu(\by_k)\big|> C.
$$
Here the first term is bounded from above by $2c=|I|$ since both
orbits have action in $I$. Hence, when $C$ is sufficiently large, we
infer that
\begin{equation}
\label{eq:2(n+1)}
\big|\hmu(\bx_k)-\hmu(\by_k)\big|> 2(n+1)
\end{equation}
and \eqref{eq:diff-index} follows. Note that here it suffices to take
\begin{equation}
\label{eq:C}
C> \frac{4N}{\lambda}(c+n+1).
\end{equation}
In fact, we can make the right-hand side in \eqref{eq:diff-index}
arbitrarily large by taking a sufficiently large $C$.

The second case is when $\ta_1=\ta_0$. Then
$$
\big|\CA_{H^{\nat k}}(\bx_k)-\CA_{H^{\nat k}}(\by_k)\big|
=\frac{\lambda}{2N}\big|\hmu(\bx_k)-\hmu(\by_k)\big|.
$$
Assume that 
\begin{equation}
\label{eq:eps1}
\eps<c_\infty.
\end{equation}
Then, by Theorem \ref{thm:energy}, $\bx_k$ and $\by_k$ can be
connected by a solution of the Floer equation only when
$$
\big|\CA_{H^{\nat k}}(\bx_k)-\CA_{H^{\nat k}}(\by_k)\big|>\eps.
$$
This,  by the first inequality of \eqref{eq:T1}, implies that

$$
\big|\CA_{H^{\nat k}}(\bx_k)-\CA_{H^{\nat k}}(\by_k)\big|>\lambda-\eps,
$$
and hence
$$
\big|\hmu(\bx_k)-\hmu(\by_k)\big|>{2N}\frac{\lambda-\eps}{\lambda}.
$$
Recall that $N\geq n+1$. Therefore, \eqref{eq:diff-index} holds when
$\eps>0$ is sufficiently small, e.g., such that
\begin{equation}
\label{eq:eps2}
\eps<\lambda/2(n+1).
\end{equation}

Finally, if $y_k$ is not one of the orbits $x_i^k$, i.e., its
normalized augmented action is not in $\tI$, we have
$$
\big|\tA_{H^{\nat k}}(y_k)-\tA_{H^{\nat k}}(x^k)\big|>k\eta>C
$$
by \eqref{eq:eta-C}. Then, exactly as in the first case, we have
\eqref{eq:2(n+1)} and hence \eqref{eq:diff-index}.
\end{proof}

Note also that the last argument shows that for any $k$-periodic orbit
$\by_k$ with action in $I$, but normalized augmented action outside
$\tI$, we have
\begin{equation}
\label{eq:mu-O}
\big|\hmu(\bx^k)-\hmu(\by_k)\big|\geq O(k)
\end{equation}
where the lower bound on the right is independent of the orbit.

Throughout the rest of the proof, we assume that $c>\lambda$, and
hence both $\bx^k$ and $\bx^k\# A_0$ have action in $I$; that $C>0$ is
sufficiently large and $\eps>0$ is sufficiently small and, in
particular, \eqref{eq:eps1} holds; and $k$ satisfying \eqref{eq:T1}
and \eqref{eq:T2} is also sufficiently large.  As a consequence, the
requirements of Lemma \ref{lemma:homology} are met.

Let $F$ be $C^2$-close to $H^{\nat k}$. Then the orbits which $x^k$
splits into under the perturbation $F$ lie in the isolating
neighborhood $U$ of $x$.

Pick a non-zero class 
$$
\gamma\in\HF_*(\bx^k)\subset \HF_*^I\big(H^{\nat k}\big).
$$
Then the class 
$$
\qq \gamma=\Phi_\alpha\big(\Phi_\beta(\gamma)\big)\in \HF_*(\bx^k\#
A_0)\subset \HF_*\big(H^{\nat k}\big)
$$
is also non-zero. (Here we use \eqref{eq:qh-fh-action} and
\eqref{eq:hom-relation}.) Moreover, by Lemma \ref{lemma:trivial}, the
intermediate class $\Phi_\beta(\gamma)$ is not in $\HF_*(\bx^k)$ or
$\HF_*(\bx_k\# A_0)$.

Now, it is a formal algebraic consequence of Lemmas
\ref{lemma:homology} and \ref{lemma:trivial} that there exists a
capped $k$-periodic orbit $\bz$ of $F$, one of the orbits that $\bx^k$
splits into, connected by a solution $u_F$ of the Floer equation to a
capped $k$-periodic orbit $\bz_*$, which, in turn, is connected to a
closed orbit $\bz'\# A_0$ by a solution $v_F$, where again $\bz'$ is
one of the orbits which $\bx^k$ splits into. Furthermore, $z_*$ is not
among the orbits arising from $x$, and $z_*$ does not enter the
neighborhood $U$ of $x$.

Passing to the limit as $F\to H^{\nat k}$ for a suitably chosen
sequence of perturbations and using the target-local compactness
theorem from \cite{Fi} exactly as in the proof of Lemma
\ref{lemma:spec}, it is easy to show that $\bx^k$ is connected by a
solution of the Floer equation for $H^{\nat k}$ to some $k$-periodic
orbit $\by_+$ lying entirely outside $U$ and such that, by Theorem
\ref{thm:energy},
$$
\CA_{H^{\nat k}}(\bx^k)-c_\infty>\CA_{H^{\nat k}}(\by_+)
$$
By passing if necessary to a subsequence of perturbations
$F\to H^{\nat k}$, we can also assume that the capped orbits $\bz_*$
converge to a capped $k$-periodic orbit $\by_*$ of $H$.

We claim that, when $k$ is sufficiently large, the orbit $y_*$ is
necessarily one of the orbits $x_i^k$. (In other words, $z_*$ is among
the orbits which $y_*=x_i^k$ splits into under the perturbation $F$.)
Indeed, this follows from \eqref{eq:mu-O} and the fact that $\by_*$
has action in $I$; for $\mu(\bz_*)=\mu(\bz)-|\beta|$ and hence
$$
\big|\hmu(\by_*)-\hmu(\bx^k)\big|\leq 2n+|\beta|.
$$

The orbit $\by_*$ might be different from $\by_+$, but
$$
\CA_{H^{\nat k}}(\by_+)\geq \CA_{H^{\nat k}}(\by_*)\geq \CA_{H^{\nat
    k}}(\bx^k\# A_0)= \CA_{H^{\nat k}}(\bx^k)-\lambda.
$$

If $c_\infty>\lambda$, we have arrived at a contradiction and the
proof is finished. Thus we can assume that $c_\infty\leq \lambda$.
Then, by \eqref{eq:T1}, we have
$$
\big|\CA_{H^{\nat k}}(\by_*)-\CA_{H^{\nat k}}(\bx^k\# A_0)\big|<\eps.
$$

However, applying the same argument to the Floer trajectories $v_F$,
we find a $k$-periodic orbit $\by_-$ of $H^{\nat k}$ which does not
enter $U$, is connected to $\bx^k\# A_0$ by a solution of the Floer
equation, and such that, again by Theorem \ref{thm:energy},
$$
\CA_{H^{\nat k}}(\by_-)>\CA_{H^{\nat k}}(\bx^k\# A_0)+c_\infty 
$$
and
$$
\CA_{H^{\nat k}}(\bx^k\# A_0)+\eps>
\CA_{H^{\nat k}}(\by_*)\geq\CA_{H^{\nat k}}(\by_-)
$$
which is impossible by \eqref{eq:eps1}.  

This contradiction concludes the proof of the theorem.
\end{proof}

\end{document}